\documentclass[12pt]{article}
\usepackage{amsmath}
\usepackage{amsfonts}
\usepackage{amssymb}
\usepackage{theorem}

\title{Generics in invariant subsets of some highly homogeneous permutation groups}
\author{M. Drzewiecka, A. Ivanov 
\thanks{Department of Applied Mathematics, Silesian University of Technology, Gliwice, Poland} and B. Mokry 
}

\newtheorem{theorem}{Theorem}[section]
\newtheorem{proposition}[theorem]{Proposition}
\newtheorem{corollary}[theorem]{Corollary}
\newtheorem{lemma}[theorem]{Lemma}
\newtheorem{definition}[theorem]{Definition}

\newtheorem{remark}[theorem]{Remark}

\newenvironment{proof}{\addvspace{8pt plus 2pt minus 2pt}\noindent\emph{Proof. }}
  { \begin{flushright}$\blacksquare$\par\addvspace{8pt plus 2pt minus
2pt}\end{flushright}}

\begin{document} 
\topmargin = 12pt
\textheight = 630pt 
\footskip = 39pt 

\maketitle

\begin{quote}
{\bf Abstract} 
Let $G$ be a closed highly homogeneous subgroup of $S_{\infty}$ not involving circular orderings.  
We show that the closure of a conjugacy class from $G$ contains a conjugacy class which is comeagre in it. 
Furthermore, we show that the family of finite partial maps extendable to elements of this conjugacy class has the cofinal amalgamation property.  
Similar statements are proved for automorphisms of typical ultrahomogeneous partial orderings. 
\end{quote}
 
\section{Introduction} 

Let $M$ be a countable ultrahomogeneous structure and let $\mathcal{C}$ be a closed subset of $\mathsf{Aut}(M)$ which is invariant under conjugacy in $\mathsf{Aut}(M)$. 
\begin{quote} 
{\em When does the set $\mathcal{C}$ contain a conjugacy class of $\mathsf{Aut}(M)$ which is comeagre in it?} 
\end{quote} 
It is clear that such $\mathcal{C}$ must have a dense conjugacy class, i.e. it must be of the following form. 
Let $\rho \in \mathsf{Aut}(M)$ and let us denote by $\mathcal{C}_{\rho}$ the least closed subset of $\mathsf{Aut}(M)$ which contains $\rho$ and is invariant with respect to conjugacy in $\mathsf{Aut}(M)$. 
It is clear that $\mathcal{C}_{\rho} = cl (\rho^{\mathsf{Aut}(M)})$ where $cl$ denotes the operation of topological closure in the automorphism group. 
The question we started with can be formulated as follows. 
{\em When does the set $\mathcal{C}_{\rho}$ contain a conjugacy class which is comeagre in it?} 
Some general condition for this can be formulated in terms of joint embedding and weak amalgamation property (WAP) introduced in \cite{iva99} (under another name) and considered later in the very influential paper \cite{KR}. 
We will describe it in Preliminaries. 
One of versions of WAP is considered in \cite{iva99} under the name Truss' condition. 
It states that a poset of finite partial automorphisms of a structure has a cofinal subfamily with the amalgamation property. 
It was introduced in \cite{truss94} and now is typically called CAP.    
{\em Is it true that in the characterization of classes $\mathcal{C}_{\rho}$ with generic automorphisms which we mentioned above, properties WAP and CAP are equivalent?} 
No example realizing WAP $\not\Rightarrow$ CAP is known (the contrary implication always holds). 
Under the additional restriction that $\mathcal{C}_{\rho} = \mathsf{Aut}(M)$ the question seems to be folklore. 
In this paper we concentrate on two questions distinguished in this introduction in the cases of highly homogeneous structures and ultrahomogeneous partial orderings. 
In particular the statements declared in the abstract will be proved. 

Due to descriptions of highly homogeneous structures and ultrahomogeneous partial orderings (see Section 1.3), the case of $\mathsf{Aut}(\mathbb{Q},<)$ becomes central in our work. 
In Section 4 we prove the key theorem of the paper stating property CAP for any $\mathcal{C}_{\rho}$ of this group (Theorem \ref{gd+rch-cap}).  
Since we do not assume the equality $\mathcal{C}_{\rho} = \mathsf{Aut}(M)$, this case involves new combinatorial challenges. 
We thus emphasize some novelty in our methods which arises from our tricks in colored linear orders (see Section 2). 

\subsection{General preliminaries}

In this paper we study the automorphism group $\mathsf{Aut}(M)$ of a countable unltrahomogeneous structure $M$. 
It is always viewed as a topological group with respect to the pointwise convergence topology. 
In particular for any partial finite isomorphism $p:M \rightarrow M$ the set 
$\mathcal{B}_p = \{ \alpha \in \mathsf{Aut}(M) \, | \, \alpha$ extends $p\}$ is clopen. 
Let $\mathcal{P}$ be the set of all finite partial isomorphisms of $M$. 
Then the family $\{ \mathcal{B}_p \, | \, p\in \mathcal{P} \}$ is a base of this topology. 

The set $\mathcal{P}$ is ordered by the relation of extension of maps, it will be denoted by $\subseteq$. 
In this terms we can formulate the standard definitions of JEP, AP, CAP and WAP. 
In the definitions below we always assume that the subset $\mathcal{P}' \subset \mathcal{P}$ is invariant under the natural action of $\mathsf{Aut}(M)$ on $\mathcal{P}$ (induced by the action on $M\times M$). 
\begin{itemize} 
\item A subfamily $\mathcal{P}'\subseteq \mathcal{P}$ has the joint embedding property if for any two elements 
$p_1, p_2 \in \mathcal{P}'$ there is $p_3$ from $\mathcal{P}'$ 
and an automorphism $\alpha \in \mathsf{Aut}(M)$ such that $p_1 \subseteq p_3$ and $\alpha (p_2) \subseteq p_3$. 
\item A subfamily $\mathcal{P}'\subseteq \mathcal{P}$ has the amalgamation property if for any $p_0 ,p_1 , p_2 \in \mathcal{P}'$ with $p_0 \subseteq p_1$ and $p_0 \subseteq p_2$ there is $p_3 \in \mathcal{P}'$ and an automorphism $\alpha \in \mathsf{Aut}(M)$ fixing $\mathsf{Dom} (p_0 ) \cup \mathsf{Rng} (p_0)$ such that $p_1 \subseteq p_3$ and $\alpha (p_2) \subseteq p_3$. 
\item A subfamily $\mathcal{P}'\subseteq \mathcal{P}$ has the cofinal amalgamation property if for any $p_0 \in \mathcal{P}'$ there is an extension $p'_0 \in \mathcal{P}'$ such that for any $p_1 , p_2 \in \mathcal{P}'$ with $p'_0 \subseteq p_1$ and $p'_0 \subseteq p_2$ there is $p_3 \in \mathcal{P}'$ and an automorphism $\alpha \in \mathsf{Aut}(M)$ fixing $\mathsf{Dom} (p'_0 ) \cup \mathsf{Rng}(p'_0)$ such that $p_1 \subseteq p_3$ and $\alpha (p_2) \subseteq p_3$. 
\item A subfamily $\mathcal{P}'\subseteq \mathcal{P}$ has the weak amalgamation property if for any $p_0 \in \mathcal{P}'$ there is an extension $p'_0 \in \mathcal{P}'$ such that for any $p_1 , p_2 \in \mathcal{P}'$ with $p'_0 \subseteq p_1$ and $p'_0 \subseteq p_2$ there is $p_3 \in \mathcal{P}'$ and an automorphism $\alpha \in \mathsf{Aut}(M)$ fixing $\mathsf{Dom} (p_0 ) \cup \mathsf{Rng} (p_0)$ such that $p_1 \subseteq p_3$ and $\alpha (p_2) \subseteq p_3$. 
\end{itemize} 
Finally we say that a subfamily $\mathcal{P}'\subseteq \mathcal{P}$ has the hereditary property (i.e. HP) if for any two elements $p_1 \in \mathcal{P}$ and $p_2 \in \mathcal{P}'$ 
the condition $p_1 \subseteq p_2$ implies $p_1 \in \mathcal{P}'$. 

We adapt the approach of \cite{iva99} to the case of groups of automophisms of {\em uniformly locally finite ultrahomogeneous structures of a finite language}. 
In particular from now on we take these assumptions on $M$. Each $p\in \mathcal{P}$ can be viewed as the expansion $(M, \mathsf{Graph}(p), \mathsf{Graph}(p^{-1}))$ and each automorphism $\alpha \in \mathsf{Aut}(M)$ can be viewed as the expansion $(M, \alpha, \alpha^{-1})$.  
Let $T_{aut}$ be the theory of the language of $M$ expanded by two functional symbols $\{ \alpha ,\beta \}$  and axiomatized by $Th (M)$, the sentence 
$\forall x(\alpha \beta (x) = \beta \alpha (x) = x)$ and all universal sentences asserting that $\alpha$ preserves the relations/functions of the language of $M$. 
Given $\rho \in \mathsf{Aut}(M)$ let $T_{\rho}$ be the extension of $T_{aut}$ by all sentences forbidding those $p\in \mathcal{P}$ for which the structure   
\[ 
(\langle \mathsf{Dom} (p) \cup \mathsf{Rng}(p)\rangle , \mathsf{Graph}(p), \mathsf{Graph}(p^{-1}))
\]  
(considered together with relations and functions of $M$) cannot be embedded as a substructure of $(M, \mathsf{Graph}(\rho), \mathsf{Graph}(\rho^{-1}))$.

\begin{lemma} \label{cl}
An automorphism $\gamma \in \mathsf{Aut}(M)$ belongs to 
$\mathcal{C}_{\rho}$ if and only if  
$(M, \gamma , \gamma^{-1}) \models T_{\rho}$.   
\end{lemma} 

This lemma is straightforward for an experienced reader.  
We now describe some hints. 
They will be helpful in the main body of the paper. 
Take an enumeration of $M^{<\omega}$, say 
$M^{<\omega}= \{ \bar{a}_0 ,\bar{a}_1 ,\ldots \}$. 
Define a metric $d$ metrizing $\mathsf{Aut}(M)$ as follows: 
\[ 
d(\gamma_1 , \gamma_2 ) = \sum \{ 2^{-n} \, : \, \mbox{ there is } a\in \bar{a}_n \mbox{ with } \gamma^{\pm 1}_1 (a) \not= \gamma^{\pm 1}_2 (a) \} .
\]
Note that 
\[
\gamma \in \mathcal{C}_{\rho} \, \Leftrightarrow \, 
\mbox{ there is a sequence } \sigma_i \in \mathsf{Aut}(M) 
\mbox{ such that } \gamma = \lim_{i \rightarrow \infty}  \sigma_i \rho \sigma^{-1}_i , 
\] 
where the limit is taken with respect to $d$.   
Furthermore, the equality 
$\gamma = \lim_{i \rightarrow \infty}  \sigma_i \rho \sigma^{-1}_i$ obviously implies that $(M, \gamma, \gamma^{-1}) \models T_{\rho}$.
To see the converse of the latter fact note that having $(M, \gamma, \gamma^{-1}) \models T_{\rho}$ we can map any restriction of $\gamma$ on 
$\langle \bigcup \{ \bar{a}_i : i\le n \}\rangle$ (considered as a finite structure) onto a substructure of $(M, \mathsf{Graph}(\rho), \mathsf{Graph}(\rho^{-1}))$.  
Using ultrahomogenity we extend this map to an automorphism of $M$, say $\sigma_i$.  
Then  
$\gamma = \lim_{i \rightarrow \infty}  \sigma_i \rho \sigma^{-1}_i$. 

It is easy to see that $T_{\rho}$ is axiomatizable by sentences which are universal with respect to the symbols $\alpha$ and $\beta$. 
In particular it satisfies the assumptions on $T$ appearing in the term ${\bf B}_T$ used in \cite{iva99} for a family of diagrams. 
Furthermore, all notions defined in \cite{iva99} for ${\bf B}_T$ (joint embedding, amalgamation, weak amalgamation,...) can be carried out in the case of $T_{\rho}$. 
In fact they become joint embedding, amalgamation, weak amalgamation,... , defined as above for finite partial isomorphisms of the following subfamily of $\mathcal{P}$. 
Let 
\begin{quote} 
$\mathcal{P}_{\rho} = \{ p \in \mathcal{P} : p$ extends to an automorphism from $\mathcal{C}_{\rho} \}$. 
\end{quote}
By Lemma \ref{cl} one can verify that the category ${\bf B}_{T_{\rho}}$ from \cite{iva99} is equivalent to the category $\mathcal{P}_{\rho}$. 

An automorphism $\gamma$ is called {\em generic} \cite{truss94} if it has a comeagre conjugacy class in $\mathsf{Aut}(M)$. 
We generalize this definition as follows. 

\begin{definition} 
Let $\mathcal{C}$ be a closed subset of $\mathsf{Aut}(M)$ which is invariant under conjugacy. 
We say that an automorphism $\gamma \in \mathsf{Aut}(M)$ is  generic in $\mathcal{C}$ if its conjugacy class $\gamma^{\mathsf{Aut}(M)}$ is comeagre in  $\mathcal{C}$. 
\end{definition} 
Reconstructing terms of \cite{iva99} it is easy to see that $\gamma$ is generic in $\mathcal{C}_{\rho}$ if and only if the structure $(M, \gamma , \gamma^{-1})$ is generic with respect to ${\bf B}_{T_{\rho}}$ in the sense of \cite{iva99}. 
Having this we can apply Theorem 1.2 from \cite{iva99} as follows: 
\begin{quote} 
{\em The set $\mathcal{C}_{\rho}$ has a generic automorphism if and only if the family $\mathcal{P}_{\rho}$ has JEP and WAP. }
\end{quote} 
It is clear that JEP is satisfied automatically: any two $p_1$ and $p_2 \in \mathcal{P}_{\rho}$ can be embedded into $\rho$. 

It is an open question if in this formulation WAP can be replaced by CAP. 
It is worth noting here that in the general context of 
Fr\"{a}iss\'{e} theory an example showing WAP $\not\Leftrightarrow$ CAP    
was obtained very recently, see \cite{KKKP}. 
This explains the major motivation of our paper. 
The main results below demonstrate that in typical situations of closed subgroups of $S_{\infty}$ one should expect that WAP and CAP are equivalent. 
Below we consider closed highly homogeneous groups and the automorphism groups of ultrahomogeneous partially ordered sets.  
\begin{remark} 
{\em 
On the other hand it is well known that JEP does not imply WAP. 
For example Corollary 3.9 of \cite{kwma} states this for the class of partial isomorphisms of finite ordered graphs and 
the class of partial isoomorphisms of finite ordered tournaments.  
We also mention here that it is well known that for any invariant subfamily $\mathcal{P}' \subseteq \mathcal{P}$ defining a closed $\mathcal{C}'\subseteq \mathsf{Aut}(M)$, the property JEP is equivalent to the property that $\mathcal{C}'$ has a dense conjugacy class, \cite{KR}. 
}
\end{remark}

\subsection{Conjugacy classes of $S_{\infty}$} 

Let us consider the topic which we have introduced above in the case of $S_{\infty}$. 

Each permutation $\rho$ of $\omega$ naturally defines a cycle function 
\[ 
f_{\rho} : (\omega \setminus \{ 0\}) \cup \{ \infty \} \rightarrow \omega \cup \{ \infty \} 
\] 
which assigns the number of cycles of $\rho$ of length $n$.  
It is well-known (and easily seen) that two permutations are conjugate if and only if their cycle functions are the same.  
In order to describe $\mathcal{C}_{\rho}$ in $S_{\infty}$ we consider two cases. 
\begin{itemize} 
\item There is a number $n$ such that each cycle of $\rho$ is bounded by $n$. 
In this case $\gamma \in \mathcal{C}_{\rho}$ if and only if 
$f_{\gamma}(\iota ) \le f_{\rho} (\iota )$ for each $\iota \in (\omega \setminus \{ 0\}) \cup \{ \infty \}$. 
\item The length of cycles of $\rho$ is not bounded.
In this case $\gamma \in \mathcal{C}_{\rho}$ if and only if 
$f_{\gamma} (n) \le f_{\rho} (n)$ for each $n\in \omega\setminus \{ 0\}$. 
\end{itemize} 
These statements are easy exercises (Lemma \ref{cl} can be applied too).  
The situation where there is a number $n$ such that each finite cycle of $\rho$ is bounded by $n$ but there are infinite cycles is included into the latter case. 

The following proposition solves the questions formulated above in the case of $S_{\infty}$. 

\begin{proposition} \label{infty} 
 For every $\rho \in S_{\infty}$ the family $\mathcal{P}_{\rho}$ has CAP and the set $\mathcal{C}_{\rho}$ has a generic permutation $\gamma$. 
Furthermore, 
\begin{itemize} 
\item $f_{\rho} (n) = f_{\gamma} (n)$ for all $n\in \omega\setminus \{ 0\}$;  
\item if the length of finite cycles of $\rho$ is not bounded then $\gamma$ does not have infinite cycles;  
\item if the length of finite cycles of $\rho$ is bounded but $\rho$ has an infinite cycle then $\gamma$ has a single infinite cycle, and if $\rho$ does not have infinite cycles then neither do $\gamma$. 
\end{itemize}  
\end{proposition} 

\begin{proof} 
If $\rho$ does not have infinite cycles, then a cofinal subfamily of  $\mathcal{P}_{\rho}$ witnessing CAP can be defined as follows. 
Let $\mathcal{P}^{c}_{\rho}$ consist of all permutations of finite subsets of $\omega$ which belong to $\mathcal{P}_{\rho}$. 
Indeed, since any $p\in \mathcal{P}_{\rho}$ extends to a conjugate of $\rho$, in order to obtain an extension of this $p$ belonging to $\mathcal{P}^c_{\rho}$ one can take a restriction of that conjugate to an appropriate finite subset of $\omega$. 
To verify AP for a triple $p_0 \subseteq p_1$ and $p_0 \subseteq p_2$ take the corresponding $p_3$ as the collection of cycles appearing in $p_1$ and $p_2$ taking each cycle as many times as the maximal number of its occurrence in $p_1$ or $p_2$.  

In the case when the length of finite cycles of $\rho$ is bounded but $\rho$ has infinite cycles let $\mathcal{P}^{c\infty}_{\rho}$ consist of all partial permutations $p\in \mathcal{P}_{\rho}$ which can be presented as disjoint union $p_c \dot{\cup} p'$ where $p_c$ is a finite permutation of a  finite subset of $\omega$ which belongs to $\mathcal{P}_{\rho}$ and $p'$ is a partial map of the form 
$\{ p'(c_i ) = c_{i+1} \, | \, i\le \ell \}$ for some $\ell$ and some set $\{ c_0 , \ldots , c_{\ell} \} \subset \omega$.   
The corresponding verification is an easy exercise. 
Now the main statement of the proposition follows from Theorem 1.2 of \cite{iva99} (it was already mentioned in the previous section). 

Let $\gamma$ be a generic permutation of $\mathcal{C}_{\rho}$. 
The first item of the second sentence of the formulation is obvious. 
To see the second item note that for any $c\in \omega$
the set of all permutations $\delta \in \mathcal{C}_{\rho}$ such that $c$ is contained in a finite cycle of $\delta$, is dense and open in  $\mathcal{C}_{\rho}$. 
In particular the set of all $\delta \in \mathcal{C}_{\rho}$
without infinite cycles is comeagre in $\mathcal{C}_{\rho}$. This set must contain $\gamma$.  

To see the last sentence we apply similar arguments. 
For example assuming existence of an infinite cycle of $\rho$ 
note that for any $a_1$ and $a_2 \in \omega$ the set of all permutations $\delta \in \mathcal{C}_{\rho}$ such that $a_1$ and $a_2$ are contained in an infinite cycle of $\delta$ or one of them is in a finite cycle, is dense and open in $\mathcal{C}_{\rho}$. 
From this one can deduce that the set of all permutations from  $\mathcal{C}_{\rho}$ with a single infinite cycle is comeagre in  $\mathcal{C}_{\rho}$. 
\end{proof} 

\subsection{Homogeneous partially ordered sets and highly homogeneous groups} 

The group $S_{\infty}$ appears as the extremal case in the two projects which we present in this section. 
In \cite{schmerl} Schmerl describes countable ultrahomogeneous partially ordered sets as follows. 
Let $1 \le n \le \omega$ and let $[n]$ be the set of natural numbers $\le n$. 
This set is viewed as an antichain. 
It is an ultrahomogeneous partially ordered set. 

Let $B_n = [n]\times \mathbb{Q}$. 
It is an ultrahomogeneous partially ordered set with respect to the ordering 
\[ 
(a,q) < (b,q') \, \Leftrightarrow \, a = b \wedge q <q'. 
\]
Let $C_n = B_n$ but the ordering is defined by 
\[ 
(a,q) < (b,q') \, \Leftrightarrow \,  q <q'. 
\]
The main theorem of \cite{schmerl} states that any countable  ultrahomogeneous partially ordered set is isomorphic to $[n]$, $B_n$, $C_n$, $1 \le n \le \omega$, or to the countable universal ultrahomogeneous partially ordered set $D$.  
\begin{quote} 
Let $M$ be a countable ultrahomogeneous partially ordered set and let $\rho \in \mathsf{Aut}(M)$. 
Does $\mathcal{C}_{\rho}$ have a generic element? 
Does the corresponding $\mathcal{P}_{\rho}$ satisfy CAP? 
\end{quote}
In the previous section we have proved that the answers are positive when $M = [\omega ]$. 

Another project is a one where $S_{\infty}$ appears as a highly homogeneous group.    
A permutation group $G$ on $\omega$ is called {\em highly homogeneous} if for any pair of finite subsets $A$ and $B \subset \omega$ of the same size there is $g\in G$  
which takes $A$ onto $B$.  
A countable structure $M$ is called highly homogeneous 
if $Aut(M)$ is highly homogeneous (under an identification with $\omega$). 
The ordering of the rationals $(\mathbb{Q},<)$ 
is an example of such a structure. 
We now define the structures $(\mathbb{Q},B)$, $(\mathbb{Q},Cr )$ and $(\mathbb{Q},S)$ which also are highly homogeneous. 
The linear betweenness relation $B(x;y,z)$ associated with $(\mathbb{Q},<)$ is defined by 
\[ 
B(x;y,z) \Leftrightarrow (y<x<z) \vee (z < x <y). 
\] 
A circular order is a twisted around total order with the natural ternary relation induced by $<$.  
The circular order on the rationals $\mathbb{Q}$ is defined by
\[ 
Cr(x,y,z) \Leftrightarrow (x<y <z) \vee (z < x <y) \vee (y<z < x). 
\]
The quaternary separation relation $S$ on $\mathbb{Q}$ is induced by the group of permutations of a circular order which preserve or reverse it. 
It is shown in \cite{Cameron} that these examples
together with $S_{\infty}$ are the only countable highly homogeneous structures. 
It is worth noting that they are ultrahomogeneous with respect to their natural languages. 
The question which were formulated in the case of partially ordered sets can be also addressed to highly homogeneous structures. 
The main results of this paper give positive answers in some cases.  

It is worth noting that the group $\mathsf{Aut}(\mathbb{Q},<)$ plays the major role in both topics defined in this section. 
This explains why in our paper we consider them together. 

\begin{remark} 
{\em 
These groups also appear together as the main source of examples in several other papers. 
For example see \cite{kasi}. 
}
\end{remark} 

\subsection{Conjugacy classes of the group of order-preserving permutations of $\mathbb{Q}$}

Since the group $\mathsf{Aut}(\mathbb{Q},<)$ will play the central role in our research, in this section we give necessary
information about this group. 

Let $\gamma\in \mathsf{Aut}(\mathbb{Q},<)$ and $x\in \mathbb{Q}$.
Then the set
\[ 
\{q\in \mathbb{Q} \, | \, (\exists m,n\in \mathbb{Z})(\gamma^n(x)\le q\le \gamma^m(x))\} 
\]
is called the {\it orbital} of $\gamma$ containing $x$.
An orbital is a singleton or an open interval. 

The parity function 
$\wp_{\gamma}:\mathbb{Q}\to \{+, -,0,\}$  of $\gamma$
is defined as follows:
\[
\wp_{\gamma}(x)=\left\{\begin{array}{r@{\quad}r}
- \quad\mbox{if}\quad \gamma (x)<x\\
0 \quad\mbox{if}\quad \gamma(x)=x\\
+ \quad\mbox{if}\quad \gamma(x)>x.
\end{array}\right.
\]
Since the parity function is constant on every orbital,
we can assign to every orbital this constant value and call
it {\it the  parity} of an orbital. 

Let $\varepsilon\in \{+, -,0\}$.
Then ${\cal O}_{\gamma}$ and ${\cal O}_{\gamma}^{\varepsilon}$ stand for the family of all orbitals of $\gamma$ and  the family of all  orbitals  of $\gamma$ having parity $\varepsilon$ respectively.\parskip0pt

For any nonempty intervals $I,J \subseteq \mathbb{Q}$  
we write $I\preceq J$ if and only if $I=J$ or for each $a\in I$ and each $b\in J$ we have $a<b$. 
We shall write $I\prec J$ if $I\preceq J \wedge I\not= J$.
The relation $\preceq$ is a partial ordering on the family
of all intervals.
For given  $\gamma\in \mathsf{Aut}(\mathbb{Q},<)$ the relation $\preceq _{\gamma}$, that is $\preceq $ relativized to ${\cal O}_{\gamma}$, is  a natural ordering of ${\cal O}_{\gamma}$.

A classical result of Schreier and Ulam (see \cite{holland})
says that
\begin{quote}
{\it $\gamma_1 ,\gamma_2$ are conjugate if and only if
$({\cal O}_{\gamma_{1}}, \preceq_{\gamma_{1}})$
and $({\cal O}_{\gamma_2}, \preceq_{\gamma_2})$ are isomorphic by an isomorphism preserving parity of the orbitals}.
\end{quote}

Recall that $\gamma \in \mathsf{Aut}(\mathbb{Q})$ is {\em generic} (i.e. its conjugacy class is comeagre in $\mathsf{Aut}(\mathbb{Q})$) if and only if for each $\varepsilon\in \{+, -,0\}$, the ordering ${\cal O}^{\varepsilon}_{\gamma}$ is without endpoints and dense in ${\cal O}_{\gamma}$ (see \cite{truss94}).

Let $p$ be a finite partial isomorphism of $(\mathbb{Q},<)$. 
By ultrahomogenity it extends to some automorphism $\gamma$ of the structure. 
The parity function  $\wp_{p}:\mathbb{Q}\to \{+, -,0\}$ 
for $p$ is obviously defined for any   
$a\in \mathsf{Dom}(p) \cup \mathsf{Rng} (p)$. 
It will be considered as a partial function.  
It is  just the restriction of $\wp_{\gamma}$ on 
$\mathsf{Dom}(p) \cup \mathsf{Rng} (p)$.

Two elements $a,b\in \mathbb{Q}$ are $p$-related if $a = b$ or one of the following cases holds: 
\begin{itemize} 
\item $a\le b\le p(a) \vee b\le a \le p(b) \vee p^{-1} (a) \le b\le a \vee p^{-1} (b) \le a \le b$, for $\wp_p(a) = +$ or $\wp_p (b) =+$. 
\item $a\le b\le p^{-1} (a) \vee b\le a \le p^{-1} (b) \vee p(a) \le b\le a \vee p (b) \le a \le b$, for $\wp_p(a) = -$ or $\wp_p (b) =-$.
\end{itemize} 
Let $\sim_p$ be the equivalence relation on $\mathbb{Q}$ generated by $p$-relatedness. 
We will call $\sim_p$-classes {\em $p$-orbitals}.    
The ordering of $\mathbb{Q}$ induces the ordering $\preceq$ of the set of $p$-orbitals. 
$p$-Orbitals intersecting $\mathsf{Dom}(p) \cup \mathsf{Rng} (p)$ have well defined parity.    
We will call them {\em colored $p$-orbitals}. 
When $p$ extends to an automorphism $\gamma$ the function $\wp_{\gamma}$ induces parity of $p$-orbitals. 
It coincides with $\wp_{p}$ on orbitals intersecting $\mathsf{Dom}(p) \cup \mathsf{Rng} (p)$ (the colored ones).

\section{Colored orderings, their homorphisms and amalgamation}

We need some preliminary material concerning colored linear orders. 
The results of this section do not depend on the main theme of the paper.  
It seems to us that they are interesting by themselves.

\subsection{Orderings colored into 3 colors and into 7 colors}

In this section we consider a linear ordering $(L, \le )$ together with a function $\kappa: L \rightarrow \Delta$ where $\Delta$ is a finite set of colors. 
We consider some cases when $|\Delta| = 3$ and $|\Delta|=7$.  
In the former case $\Delta = \{ +,-,0\}$. 

\begin{definition} \label{chi}
Let $(L, \le , \kappa )$ be a colored linear ordering where 
$\kappa :  L\rightarrow \{ +, -, 0 \}$. 
Let $L_0$ be a convex subset of $L$ and $\varepsilon \in \{ +,-,0\}$. 
\begin{itemize} 
\item We say that $L_0$ is of type $\infty_{+,-,0}$ if for each natural number $n$ there is a sequence $a_1 < a_2 < \ldots < a_{3(n+1)}$ in $L_0$ such that  
\[ 
\kappa (a_{3i +1}) = + \, , \, \kappa (a_{3i +2}) = - \, ,  \, \kappa (a_{3i +3}) = 0 \, \mbox{ where } i \le n.  
\] 
\item  We say that $L_0$ is of type $\varepsilon$ if $\kappa (L_0)= \{ \varepsilon\}$ and, moreover, in the case $\varepsilon =0$ the set $L_0$ consists of a single element. 
\item We say that $L_0$ is of type $\infty_0$ if $L_0$ is infinite and  $\kappa (L_0)= \{ 0\}$. 
\item Let $\varepsilon' \in \{ +,-,0\} \setminus \{ \varepsilon\}$. 
We say that $L_0$ is of type 
$\infty_{\varepsilon \varepsilon'}$ if 
$\kappa (L_0 ) = \{ \varepsilon, \varepsilon' \}$ and for each natural number $n$ there is a sequence $a_1 < a_2 < \ldots < a_{2(n+1)}$ in $L_0$ such that  
\[ 
\kappa (a_{2i +1}) = \varepsilon \, , \, \kappa (a_{2i +2}) = \varepsilon' , \, \mbox{ where } i \le n.  
\] 
\end{itemize} 
\end{definition}

Note that the types $\infty_{\varepsilon \varepsilon'}$ and $\infty_{\varepsilon'\varepsilon}$ coincide. 
It will be convenient to use the pair $+-$, $+0$ and $-0$ in the notation. 
In particular in Definition \ref{chi} eight types appear. 
Seven of them will be used as colors later. 

\begin{proposition} \label{p-epsilons}
Let $(L, \le )$ be an  infinite linearly ordered set and let 
$\kappa: L\rightarrow \{ +, -, 0 \}$ be a function coloring  
$L$ into three colors $+,-,0$. 
Assume that $L$ is not of type $\infty_{+,-,0}$.  
Then the ordering $(L, \le )$ can be decomposed into finitely many intervals 
$L = L_1 \, \dot{\cup}\,  L_2 \, \dot{\cup} \, \ldots \dot{\cup} \, L_m$ with the natural  ordering 
$L_1 \prec L_2 \prec \ldots \prec L_m$ such that $|\kappa (L) | \le 2$ for any $i\le m$.  
\end{proposition} 

\begin{proof} 
Choose the minimal $n$ such that the situation of the definition of type $\infty_{+,-,0}$ cannot be realized. 
Choose a sequence 
$a_1 < a_2 < \ldots < a_{3n}$ where 
\[ 
\kappa (a_{3i +1}) = + \, , \, \kappa (a_{3i +2}) = - \, ,  \, \kappa (a_{3i +3}) = 0 \, \mbox{ for } i < n.  
\] 
We will assume that $n>0$. 
The case $n=0$ with  $|\kappa (L) | = 3$ can be treated in the same way. 
Instead of a sequence af $a_i$-s we just fix one $a_1$ and apply the arguments below. 

Let $[-\infty, a_{1}]$ denote the initial interval of $L$ of all $x\le a_1$.  
We claim that 
\begin{quote} 
each interval 
$[-\infty, a_{1}], [a_{1},a_{2}], [a_{2},a_{3}], \ldots ,[a_{3n}, +\infty]$ can be decomposed into  at most three $(\le 2)$-colored intervals. 
\end{quote} 
Let $i<n$ and let $A= [a_{3i+1},a_{3i+2}]$. 
We may assume that $|A| >2$. 
Take the maximal initial subinterval of $A$, say $A_1$, such that $- \not\in \kappa(A_1 )$. 
Let $A_4$ be the maximal final subinterval of $A$ such that $+ \not\in \kappa (A_4 )$. 
Decompose $A= A_1 \dot{\cup} A_2$ and $A= A_3 \dot{\cup} A_4$. 
Let $A' =A_2 \cap A_3$.  

If $A'=\emptyset$ then $A=A_1 \, \dot{\cup} \, (A\setminus A_1 )$ is a required decomposition.   
If $A'\not= \emptyset$, then $A = A_1 \, \dot{\cup} \, A' \, \dot{\cup} \, A_4$ and 
\begin{quote} 
for each $a\in A'$ there are $a' \in A'$ and $a''\in A'$ such that $a'\le a \le a''$ and $\kappa (a') = - \wedge \kappa (a'') = +$. 
\end{quote} 
If there is $a\in A'$ with $\kappa (a) = 0$ then adding it to the sequence of $a_i$-s together with corresponding $a'$ and $a''$ we get a contradiction with the choice of $n$.  
As a result we see that the sets $\kappa (A_1)$, $\kappa (A')$ and $\kappa (A_4 )$ are of size $<3$. 

The cases $A= [a_{3i+2},a_{3i+3}]$ and $A= [a_{3i},a_{3i+1}]$ are similar.  
In the case $A= [-\infty ,a_{1}]$ we decompose $A= A_1 \dot{\cup} A_2$ and $A= A_3 \dot{\cup} A_4$ so that $A_1$ is the maximal initial subinterval of $A$ such that $+ \not\in \kappa(A_1 )$ and $A_4$ is the maximal final subinterval of $A$  such that $0 \not\in \kappa(A_4 )$ (it can happen that $A_1 =\emptyset$). 
Then we repeat the argument above. 
The case $A = [a_{3n}, +\infty]$ is dual. 
\end{proof} 

\begin{lemma} \label{l-epsilons}
Let $(L, \le )$ be an  infinite linearly ordered set and let 
$\kappa: L\rightarrow \{ \sharp_1, \sharp_2 \}$ be a function coloring $L$ into two colors. 
Assume that there is a natural number $n$ such that the set $L$ does not have a sequence $a_1 < a_2 < \ldots < a_{2(n+1)}$ where 
\[ 
\kappa (a_{2i +1}) = \sharp_1 \, , \, \kappa (a_{2i +2}) = \sharp_2 \, ,  \mbox{ with } i \le n.  
\] 
Then the ordering $(L, \le )$ can be decomposed into finitely many intervals 
$L = L_1 \, \dot{\cup}\,  L_2 \, \dot{\cup} \, \ldots \dot{\cup} \, L_m$ with 
$L_1 \prec L_2 \prec \ldots \prec L_m$ such that $|\kappa (L) | = 1$ for any $i\le m$.  
\end{lemma} 

\begin{proof} 
We use the same strategy of the proof as in Proposition \ref{p-epsilons}. 
Choose the minimal $n$ satisfying the condition of the formulation. 
We may assume that $n>0$. 
Then:  
\begin{quote} 
each interval  $[-\infty, a_{1}], [a_{1},a_{2}], [a_{2},a_{3}], \ldots ,[a_{2n}, +\infty]$ can be decomposed into  at most two monochromatic intervals. 
\end{quote} 
For example consider $A= [a_{2i+1}, a_{2i +2}]$ with $i<n$. 
Let $A_1$ be the maximal initial interval with $\kappa (A_1) =\sharp_1$.  
If $|\kappa (A\setminus A_1)|=2$ then $A\setminus A_1$ contains a pair $a' < a''$ with $\kappa (a') = \sharp_2$ and $\kappa (a'') = \sharp_1$. 
This contradicts the choice of $n$. 
The remaining cases are similar. 
\end{proof}

We now introduce new colors and some partial ordering of them. 
Let 
\[ 
\chi = \{ \{ + \} , \{ - \} , \{ 0 \} , \infty_0 \} \cup 
\{ \infty_{\varepsilon_1\varepsilon_2} \, | \, \varepsilon_1  \varepsilon_2 \in \{ +-,+0, -0\} \} . 
\] 
The ordering of $\chi$ corresponds to the relation $\subseteq$. 
In particular when 
$\varepsilon \in \{ \varepsilon_1 , \varepsilon_2 \}$ we put $\{ \varepsilon \} \subset \infty_{\varepsilon_1 , \varepsilon_2 }$. 
We also put $\{ 0 \} \subset \infty_0 \subset \infty_{+0}$ 
and $\infty_0 \subset \infty_{-0}$. 
Note that 
\[ 
|\chi | = 7.  
\] 

Let $(L, \le, \kappa )$ be a $\{ +,-,0\}$-colored linear ordering. 
To each decomposition 
$L = \dot{\bigcup} \{ L_i \,  |  \, i\le n \}$ into finitely many intervals of types 
\[ 
\tau \in \{ + ,  -  ,  0  , \infty_0 \} \cup 
\{ \infty_{\varepsilon_1\varepsilon_2} \, | \, \varepsilon_1  \varepsilon_2 \in \{ +-,+0, -0\} \}  
\] 
we associate a $\chi$-coloring $\hat{\kappa}$ of the family 
$\{ L_i \, | \, i \le n \}$ as follows:  
\begin{quote} 
$\hat{\kappa}(L_i) = \{ \tau \}$ if $L_i$ is of type $\tau \in \{ + ,  -  ,  0 \}$ \\ 
$\hat{\kappa}(L_i) = \tau \, \,\, \, $ if $L_i$ is of type 
$\tau \in \{ \infty_0 \} \cup \{ \infty_{\varepsilon_1\varepsilon_2} \, | \, \varepsilon_1  \varepsilon_2 \in \{ +-,+0, -0\} \}$.
\end{quote} 
As a result the family $\mathcal{L}= \{ L_i \, | \, i \le n \}$ becomes a finite $\chi$-colored ordered set where the ordering is just $\prec$. 

We use the structure of $\chi$ in the following definition. 

\begin{definition} \label{d-epsilons} 
Let $(\mathcal{L}, <, \hat{\kappa} )$ be a finite $\chi$-colored linear ordering. 
We say that $(\mathcal{L}, <, \hat{\kappa} )$ is canonical if 
for any $x,y \in \mathcal{L}$ the following conditions are saidfied: 
\begin{itemize} 
\item if $x\not= y$ and 
$\hat{\kappa} (x) \subseteq \hat{\kappa} (y )$ then  
there is $u \in \mathcal{L}$ such that $x< u < y$ or $y < u <x$.  
\item 
if there is $u \in \mathcal{L}$ such that $x< u < y$ or $y < u <x$ and for some 
$\varepsilon_1 \varepsilon_2\not=  \varepsilon'_1 \varepsilon'_2$ we have 
$\hat{\kappa} (x) =\infty_{\varepsilon_1 \varepsilon_2}$ and 
$\hat{\kappa} (y) =\infty_{\varepsilon'_1 \varepsilon'_2}$ then one of the following statements holds:  

(a) there are $z_1 , z_2 \in \mathcal{L}$ such that $x< z_1 <z_2 < y$ or $y < z_1 < z_2 <x$ and 
$\hat{\kappa} (z_1) , \hat{\kappa} (z_2) \in \{ \{ +\}, \{ - \}, \{ 0 \} , \infty_0 ,\infty_{\varepsilon_1 \varepsilon_2}, \infty_{\varepsilon'_1 \varepsilon'_2} \}$ and $\hat{\kappa} (z_1) \not= \hat{\kappa} (z_2)$;  

(b) there is $z$ such that  
$\hat{\kappa} (z)$ is of the form 
$\infty_{\varepsilon \varepsilon'}$ where 
$\varepsilon \not\subseteq \hat{\kappa} (x)\wedge 
\varepsilon'\not\subseteq \hat{ \kappa}(y)$ and  
$(x< z < y) \vee (y < z <x)$. 
\end{itemize} 
\end{definition} 

The following theorem is the final result of this section. 

\begin{theorem} \label{decomposition} 
Let $(L, \le, \kappa )$ be a $\{ +,-,0\}$-colored linear ordering. 
Assume that $\mathcal{L}$ is not of type $\infty_{+,-,0}$. 
Then there is a decomposition 
$L = L_1 \, \dot{\cup} \,  L_2 \, \dot{\cup} \ldots  \dot{\cup} \, L_n$ into finitely many intervals such that the corresponding $\chi$-colored ordering 
$(\mathcal{L},\preceq ,\hat{\kappa})$ is canonical. 
Furthermore, for any two finite decompositions of $L$ defining canonical $\chi$-colored orderings, these orderings are isomorphic. 
\end{theorem} 

\begin{proof} 
Let us find a canonical decomposition. 
Take any $\varepsilon_1 \varepsilon_2 \in \{ +-, +0, -0 \}$. 
Let 
$L^{\varepsilon_1 \varepsilon_2}_1 , \ldots , L^{\varepsilon_1 \varepsilon_2}_i , \ldots$ be the family of all maximal subintervals of $L$ of type 
$\infty_{\varepsilon_1 \varepsilon_2}$. 
It obviously consists of pairwise disjoint intervals, i.e. every two intervals 
$L^{\varepsilon_1 \varepsilon_2}_i$ and $L^{\varepsilon_1 \varepsilon_2}_j$ are separated by a point of the color distinct from $\varepsilon_1$ and $\varepsilon_2$. 
By Proposition \ref{p-epsilons} this family is finite. 

Consider the set 
\[ 
L^{\perp} = L \setminus \bigcup \{ L' \, | \, L' 
\mbox{ coincides with some } L^{\varepsilon_1 \varepsilon_2}_i 
\mbox{ for some }\varepsilon_1 \varepsilon_2 \in\{ +-, +0, -0\} \} . 
\] 
Applying Lemma \ref{l-epsilons} to subintervals of $L^{\perp}$ we see that $L^{\perp}$ can be decomposed into a union of finitely many monochromatic intervals. 

Assume 
$\{\varepsilon_1 ,\varepsilon_2 \} \not= \{\varepsilon'_1 ,\varepsilon'_2\}$. 
Then $L^{\varepsilon_1 \varepsilon_2}_i \cap L^{\varepsilon'_1 \varepsilon'_2}_j$ is either monochromatic or empty. 
In the former case to reach disjointness one can assign the monochromatic part to only one of them. 
Then we do not insist on maximality of the another one. 
If $L^{\varepsilon_1 \varepsilon_2}_i$ and $L^{\varepsilon'_1 \varepsilon'_2}_j$ are not neighbours and 
$L^{\varepsilon_1 \varepsilon_2}_i \prec L^{\varepsilon'_1 \varepsilon'_2}_j$, then 
they should be separated by a pair of points, say $z_1 < z_2$ 
where $\kappa (z_1) \not\in \{\varepsilon_1 ,\varepsilon_2\}$ and $\kappa (z_2 ) \not\in \{\varepsilon'_1 ,\varepsilon'_2\}$.  

Now let $(\mathcal{L}, \prec, \hat{\kappa} )$ be the ordered set of all intervals of the form $L^{\varepsilon_1 \varepsilon_2}_i$ (after the correction above) and all maximal monochromatic intervals forming $L^{\perp}$.   
Let $\hat{\kappa}$ be the $\chi$-coloring constructed accordingly to the procedure given before Definition \ref{d-epsilons}. 
It is easy to see that it is canonical. 

Let us prove that a canonical ordering is defined uniquely. 
Assume that 
$L = I_1 \, \dot{\cup} \,  I_2 \, \dot{\cup} \ldots  \dot{\cup} \, I_m$ is a decomposition of $L$ giving a canonical $\chi$-ordering. 
Since the decomposition is finite for any $L^{\varepsilon_1 \varepsilon_2}_i$ there is $I_j$ such that 
$L^{\varepsilon_1 \varepsilon_2}_i \cap I_j$ is also of type $\infty_{\varepsilon_1 \varepsilon_2}$. 
If that $I_j$ is not maximal then $I_{j+1}$ (or $I_{j-1}$) is of type $\infty_{\varepsilon'_1 \varepsilon'_2}$ with 
$\{ \varepsilon_1 ,\varepsilon_2\}\cap \{\varepsilon'_1 ,  \varepsilon'_2 \} \not=\emptyset$, i.e. extending $I_j$ by a monochromatic subinterval of $I_{j+1}$ (resp. $I_{j-1}$) we obtain a maximal one. 
Moreover, as we already know, $I_{j+1}$ (resp. $I_{j-1}$) has an intersection with some $L^{\varepsilon'_1 \varepsilon'_2}_k$ of type $\infty_{\varepsilon'_1 \varepsilon'_2}$. 
This argument leads us to the observation that $L^{\perp}$ defined in the first part of the proof coincides with  
\[ 
L \setminus \bigcup \{ I_j \, | \, I_j 
\mbox{ is of type } \infty_{\varepsilon_1 \varepsilon_2} 
\mbox{ for some }\varepsilon_1 \varepsilon_2 \in\{ +-, +0, -0\} \} . 
\] 
On the other hand it is easy to see that a canonical decomposition of $L^{\perp}$ is  unique. 
\end{proof}

\subsection{CAP for colored orderings} 

We study categories of colored orderings with respect to some special homomorphisms. 
This is why we start this section by two following definitions. 

\begin{definition} \label{kappa-hom} 
Let $(L, <, \kappa )$ and $(L', <', \kappa' )$ be $\{ +,-,0\}$-colored linear orderings. 
A map $\phi: L \rightarrow L'$ 
is called a $\kappa$-homomorphism if for any 
$x_1 ,x_2 ,x \in L$  
\[
\kappa (x ) = \kappa' (\phi(x)) , 
\]
\[ 
x_1 \le x_2 \Rightarrow \phi (x_1 ) \le' \phi (x_2 ) \,  
\] 
and $\phi$ is injective on the set of all elements of color $0$.   
\end{definition}

\begin{definition} \label{chi-hom} 
Let $(\mathcal{L}, <, \hat{\kappa} )$ and $(\mathcal{L}', <', \hat{\kappa}' )$ be $\chi$-colored linear orderings. 
A map $\phi: \mathcal{L} \rightarrow \mathcal{L}'$ 
is called a $\chi$-homomorphism if for any 
$x_1 ,x_2 ,x \in \mathcal{L}$  
\[ 
x_1 \le x_2 \Rightarrow \phi (x_1 ) \le' \phi (x_2 ) \, \mbox{ and } 
\] 
\[
\hat{\kappa} (x) \subseteq \hat{\kappa}' (x ) ). 
\] 
\end{definition} 

\begin{remark} \label{phi-inj}
{\em 
It is worth noting that when $(\mathcal{L}, <, \hat{\kappa} )$ in Definition \ref{chi-hom} is a canonical $\chi$-colored linear orderings and $\phi (x) = \phi (y)$ for some $x,y \in \mathcal{L}$, then $x$ and $y$ are monochromatic of distinct colors and 
$\hat{\kappa} (x) \cup \hat{\kappa}(y) = \hat{\kappa}' (\phi (x))$ or $x$ and $y$ are monochromatic of the same color and there is some $u$ between them such that  
$\hat{\kappa} (x) \cup \hat{\kappa} (u) \cup \hat{\kappa}(y) = \hat{\kappa}' (\phi (x))$ where $\hat{\kappa}' (\phi (x))$ is not monochromatic. 
Indeed, according to Definition \ref{d-epsilons} when $x$ and $y$ have compatible colors they should be separated by an element of an incompatible color.  
On the other hand when $x$ and $y$ have incompatible colors and one of them is not monochromatic then $\kappa'(\phi (x)) \not=\kappa'(\phi (y))$.}
\end{remark}

Let $(L, <, \kappa )$ be an infinite, countable $\{ +,-,0\}$-colored linear ordering which is not of type $\infty_{+,-,0}$.
According to Theorem \ref{decomposition} there is a decomposition 
\[ 
L = L_1 \, \dot{\cup} \,  L_2 \, \dot{\cup} \ldots  \dot{\cup} \, L_n 
\] 
into finitely many intervals such that the corresponding $\chi$-colored ordering 
$(\mathcal{L},\preceq ,\hat{\kappa})$ is canonical. 
Let us fix this $n$. 

Let $\mathcal{K}$ be the age of $(L, <, \kappa )$, i.e. the family of all finite substructures of $(L, <, \kappa )$. 
We consider $\mathcal{K}$ together with the family of all $\kappa$-homomorphisms. 
Amalgamation of members of $\mathcal{K}$ is considered in this category. 

\begin{definition} \label{rch}
Let $\mathcal{K}^{rch} \subset \mathcal{K}$ be the family of all members $(A, < , \kappa )\in \mathcal{K}$  
such that for any $L_i \in \mathcal{L}$ 
\begin{itemize} 
\item  the intersection $A\cap L_i$ is not empty; 
\item if $L_i$ is of type $\{ + \}$, $\{ - \}$ or $\{ 0 \}$, then $A\cap L_i$ is single;  
\item if $L_i$ is of type $\infty_{\varepsilon_1 \varepsilon_2}$ or $\infty_{0}$ then $A\cap L_i$ consists of an increasing sequence of size at least $n^n$ of pairs of the color $\varepsilon_1 , \varepsilon_2$ (resp. sequence of size at least $n^n$ of elements of color $0$).  
\end{itemize} 
\end{definition}

It is clear that $\mathcal{K}^{rch}$ is cofinal in $\mathcal{K}$ under $\kappa$-homomorphisms: any member of $\mathcal{K}$ maps into a member of $\mathcal{K}^{rch}$. 
The main theorem of this section states that $\mathcal{K}^{rch}$ witnesses CAP for $\mathcal{K}$. 

\begin{theorem} \label{rch-amal}
The subfamily $\mathcal{K}^{rch} \subset \mathcal{K}$
has the amalgamation property. 
\end{theorem} 

\begin{proof} 
If $\mathcal{L}$ does not have intervals of type $\infty_{\varepsilon_1 \varepsilon_2}$ and $\infty_{0}$ then $\mathcal{K}^{rch}$ consists of a single structure and the statement of the theorem is obvious. 
From now on in this proof we assume that there are intervals of type $\infty_{\varepsilon_1 \varepsilon_2}$ or $\infty_{0}$.  

We start the general case with an easy observation (similar to Remark \ref{phi-inj}) that canonicity of $\mathcal{L}$ guarantees that $\kappa$-homomorphisms of members of $\mathcal{K}^{rch}$ are injective. 

We now need the following claim. 
\bigskip 

\noindent 
{\bf Claim.} 
{\em Let $(A, <, \kappa )$ and $(B, <, \kappa )$ be $\{ +,-,0\}$-colored linear orderings from $\mathcal{K}^{rch}$. 
Then any $\kappa$-homomorphism $\phi: A \rightarrow B$ 
preserves the intervals $L_i$ of the canonical decomposition $\mathcal{L}$. 
}

Let $L_i$ be of type $\infty_{\varepsilon_1 \varepsilon_2}$.   
By the pigeon hole principle there are at least $n$ pairs of the color $\varepsilon_1 \varepsilon_2$ in $A\cap L_i$ which are mapped to the same $B\cap L_k$ for some $L_k$ of type $\infty_{\varepsilon_1 \varepsilon_2}$. 
Since $\mathcal{L}$ is finite and any two $A \cap L_i$ and $A\cap L_j$ of type $\infty_{\varepsilon_1 \varepsilon_2}$ are separated in $A$ by an element of the color outside of $\{\varepsilon_1 ,\varepsilon_2\}$, then $L_i = L_k$ and, moreover, the whole $A\cap L_i$ is mapped to $B\cap L_i$. 
The same argument works for intervals $L_i$ of type $\infty_{0}$. 

If $L_i$ is of type $\{ + \}$, $\{ - \}$ or $\{ 0 \}$, then $A\cap L_i$ is single and the neighbors of $L_i$ in $\mathcal{L}$ are of incompatible colors. 
This guarantees the statement of the claim. 
\bigskip 

Let us verify the amalgamation property. 
Let $(A, <, \kappa )$, $(B, <, \kappa )$ and $(C, <,\kappa )$ be $\{ +,-,0\}$-colored linear orderings from $\mathcal{K}^{rch}$ and let $\phi_1 :  A\rightarrow B$ and $\phi_2 : A \rightarrow C$ be $\kappa$-homomorphisms. 
Using the claim above we see that it suffices to verify the amalgamation property only for substructures of the form $A \cap L_i$, $B\cap L_i$ and $C\cap L_i$ where $L_i$ is of type $\infty_{\varepsilon_1 \varepsilon_2}$ or $\infty_{0}$. 
The latter case is obvious. 
To see the former one we apply the disjont amalgamation:  
let 
\[ 
D\cap L_i = (B\cap L_i ) \cup_{\phi_1(A\cap L_i) = \phi_2 (A \cap L_i)} (C\cap L_i). 
\]
To get a linear ordering of $D\cap L_i$ we use the following rule.  
Let $a'\in A$ be the successor of $a \in A$ with respect to the ordering of $A$. 
Let the interval $(a,a')$ in $B$ be the sequence  
$a < b_1 < \ldots < b_s < a'$ and let $(a,a')$ in $C$ be 
$a < c_1 < \ldots < c_t < a'$. 
Then we present $(a,a')$ in $D$ by 
\[ 
a < b_1 < \ldots < b_s < c_1 < \ldots < c_t < a' . 
\] 
Since $\kappa (a) \not= \kappa (a' )$, the numbers $s$ and $t$ are even. 
Thus $\kappa (a) = \kappa (b_s ) = \kappa (c_t)$ 
and $\kappa (a') = \kappa (b_1 ) = \kappa (c_1 )$. 
As a result we see that $D\cap L_i$ is of the same type with $A\cap L_i$, i.e $\infty_{\varepsilon_1 \varepsilon_2}$.

\end{proof}

\section{The group $\mathsf{Aut}(\mathbb{Q},<)$ and classes $\mathcal{C}_{\rho}$}

In this section we study some closed invariant subsets of 
$\mathsf{Aut}(\mathbb{Q},<)$. 
We denote by $\mathcal{P}$ the set of all finite partial isomorphisms of $(\mathbb{Q},<)$. 
When $\rho \in \mathsf{Aut}(\mathbb{Q},<)$ let 
$\mathcal{C}_{\rho}$ be the topological closure of $\rho^{\mathsf{Aut}(\mathbb{Q},<)}$ and 
\begin{quote} 
$\mathcal{P}_{\rho} = \{ p \in \mathcal{P} : p$ extends to an automorphism from $\mathcal{C}_{\rho} \}$. 
\end{quote}

\subsection{The general case} 

Given  $\gamma\in \mathsf{Aut}(\mathbb{Q},<)$ the ordered set of orbitals $({\cal O}_{\gamma}, \prec_{\gamma} )$ is $\{ +,-,0 \}$-colored by the parity function $\wp_{\gamma}$.
Assuming that this ordering is not of type $\infty_{+,-,0}$ one can apply Theorem  \ref{decomposition}. 
This finds a canonical $\chi$-colored ordering $(\mathcal{L}_{\gamma}\prec_{\gamma}, \hat{\wp}_{\gamma})$. 

Let $p$ be a partial automorphism of $(\mathbb{Q},<)$ and $\sim_p$ be the partial orbital equivalence relation. 
If $A\subseteq \mathbb{Q}$ and $p = \gamma \upharpoonright A$,  the parity function 
$\wp_{\gamma}:\mathbb{Q}\to \{+, -,0\}$ defines a $\{ +,-,0\}$-colored ordered set $(A/\sim_{p}, \prec_p, \wp_p )$.

\begin{theorem} \label{equiva}
Let $\rho ,\gamma \in \mathsf{Aut}(\mathbb{Q}, \le )$. 
The following conditions are equivalent. 
\begin{enumerate}
\item $\gamma \in \mathcal{C}_{\rho}$; 
\item for any finite $A\subseteq \mathbb{Q}$ there is a $\wp$-homomorphism   $\phi : (A/\sim_{p}, \prec_p ,\wp_{p}) \rightarrow (\mathcal{O}_{\rho}, \prec_{\rho}, \wp_{\rho})$ where $p = \gamma \upharpoonright A$; 
\item $\mathcal{O}_{\rho}$ is of type $\infty_{+,-,0}$ or both $\mathcal{O}_{\gamma}$ and $\mathcal{O}_{\rho}$ are not of type $\infty_{+,-,0}$ and there is a $\chi$-homomorphism from the canonical decomposition $(\mathcal{L}_{\gamma}, \prec_{\gamma}, \hat{\wp}_{\gamma})$ corresponding to $(\mathcal{O}_{\gamma}, \prec_{\gamma}, \wp_{\gamma} )$ to the canonical decomposition $(\mathcal{L}_{\rho}, \prec_{\rho}, \hat{\wp}_{\rho})$ corresponding to $(\mathcal{O}_{\rho}, \prec_{\rho}, \wp_{\rho} )$. 
\end{enumerate}   
\end{theorem}

\begin{proof} 
The implication $1 \rightarrow 2$ is easy. 
Indeed, for any finite $A\subset \mathbb{Q}$ there is an automorphism $\sigma$ such that $\sigma^{-1} \rho \sigma$ coincides with $\gamma$ on $A$. 
Let $p = \gamma \upharpoonright A$ and 
$(A\cup p(A))/\sim_p = \{ C_1 ,\ldots ,  C_k \}$ where $C_1 \prec \ldots \prec C_k$.  
Then the automorphism $\sigma$ maps each $C_i$ into some orbital of $\mathcal{O}_{\rho}$ so that the parity is preserved. 

To see $2 \rightarrow 1$ we take an increasing sequence of finite sets $A_i$ such that $\mathbb{Q} = \bigcup A_i$.  
Let $p_i = \gamma \upharpoonright A_i$. 
It suffices to show that for every $i$ there is an automorphism $\sigma_i$ such that $p_i = \sigma^{-1}_i \rho \sigma_i \upharpoonright A_i$. 

Let $(A_i \cup p_i (A_i ))/\sim_{p_i} = \{ C_1 ,\ldots , C_k ,\ldots ,C_s \}$ where $C_1 \prec \ldots \prec C_s$ is the corresponding set of $p_i$-orbitals. 
Applying the second condition of the theorem we find a $\wp$-homomorphism $\phi$ from $(A_i \cup p_i (A_i ))/\sim_{p_i})$ to $\mathcal{O}_{\rho}$. 
In order to construct $\sigma_i$ we define for every $k\le s$  an equivariant embedding of $(C_k , <, p_i \upharpoonright C_k )$ into the $\rho$-orbital $(\phi (C_k ),<,\rho \upharpoonright \phi (C_k ))$.   
It will be organized in such a way that all these embeddings together form a finite partial isomorphism mapping $(A_i \cup p_i (A_i ), < , p_i )$ to $(\mathbb{Q}, < , \rho )$. 
By ultrahomogenity of $(\mathbb{Q},<)$ it extends to an automorphism. 
The latter one is our $\sigma_i$.  

If the parity of $C_k$ is $0$ then for some $q\in \mathbb{Q}$, $C_k = \{ q \}$ and $p_i (q) = q$. 
Thus the orbital $\phi (C_k )$ consists of one element and 
$\phi$ defines the required embedding. 
It is worth noting here that according to Definition \ref{kappa-hom} for another $C_{\ell}$ of parity $0$ the orbitals $\phi (C_k )$ and $\phi (C_{\ell} )$ are distinct. 
Thus the corresponding embeddings define a partial equivariant isomorphism on $p_i$-orbitals of parity $0$. 

Let us consider the case when  the parity of $C_k$ is $+$. 
In this case $C_k$ has a natural decomposition into partial $p_i$-orbits. 
We may choose their representatives, say $a_1 < a_2 < \ldots < a_t$. 
Furthermore, we assume that those $a_j$ are minimal representatives of these orbits (say of sizes $l_j$). 
Using induction on $j \le t$ one can construct an isomorphic embedding of 
\[ 
\{ a_1 , p_i (a_1 ), \ldots p^{l_1}_i (a_1 ),a_2 , p_i (a_2 ), \ldots ,  a_j , p_i (a_j ), \ldots ,p^{l_j}_i (a_j )\} 
\] 
into the orbital $\phi (C_k )$ (which is also of parity $+$, i.e. is infinite) so that the action of $p_i$ agrees with the action of $\rho$. 
Finaly after $t$-th step we get an isomorphic embedding of $(C_k , < , p_i \upharpoonright C_k )$ into $(\phi (C_k ), <, \rho \upharpoonright \phi (C_k ))$. 

It can happen that $\phi (C_k ) = \phi (C_{k+1})$. 
Thus having constructed an embedding of $(C_k , < , p_i \upharpoonright C_k )$ into $(\phi (C_k ), <, \rho \upharpoonright \phi (C_k ))$ we repeat the procedure for 
$(C_{k+1} , < , p_i \upharpoonright C_{k+1} )$ starting with the image for the least element of $C_{k+1}$ higher than 
the elements corresponding to $C_k$. 
The case when the parity of $C_k$ is $-$, is similar. 

Let us prove $3 \rightarrow 2$. 
Assume that both $\gamma$ and $\rho$ are not of type $\infty_{+,-,0}$ and there is a $\chi$-homomorphism, say $\hat{\psi}$, from the canonical decomposition $(\mathcal{L}_{\gamma}, \prec, \hat{\wp}_{\gamma})$  
to the canonical decomposition $(\mathcal{L}_{\rho}, \prec, \hat{\wp}_{\rho})$.  
Let $A$ be a finite subset of $\mathbb{Q}$ and let $L$ be an element of $\mathcal{L}_{\gamma}$. 
Remember that $L$ is a set of orbitals of $\gamma$. 
Define $A_L = A \cap (\bigcup \{ O \, | \, O \in L\})$. 
In order to define a $\wp$-homomorphism from 
$(A_L /\sim_{\gamma})$ to the $\prec$-ordering of $\hat{\psi}(L)$ we consider several cases depending on the $\chi$-color of $L$. 

If $\hat{\wp}_{\gamma} (L) \in \{ \{+\}, \{ - \}\}$ then according to Definition \ref{chi-hom} the $\chi$-homomorphism 
$\hat{\psi}$ maps $L$ to an interval of orbitals containing orbitals of parity $\hat{\wp}_{\gamma}(L)$.  
Thus the elements of $(A_L /\sim_{\gamma})$ can be mapped 
to the infinite orbital of $\rho$ of the corresponding parity. 

If $\hat{\wp}_{\gamma} (L) \in \{ \infty_{\varepsilon_1 \varepsilon_2} \, | \, \varepsilon_1 , \varepsilon_2 \in \{ +,-,0 \}\}$ then according to Definition \ref{chi-hom} the $\chi$-homomorphism $\hat{\psi}$ maps $L$ to an interval of orbitals of the same type. 
Thus to any $C\in (A_L /\sim_{\gamma})$ one can assign an orbital of $\rho$ of the same parity.  
Since  $\hat{\psi}(L)$ consists of infinitely many orbitals 
this can be arranged by an injective $\prec$-preserving map. 

If $\hat{\wp}_{\gamma} (L) \in \{ \{0\}, \infty_{0}\}$ then 
the elements of $(A_L /\sim_{\gamma})$ can be mapped 
injectively to the set of orbitals of $\rho$ of parity $0$ which belong to $\hat{\psi} (L)$.  

We apply the arguments above to any $L \in \mathcal{L}_{\gamma}$. 
As a result we get a required $\wp$-homomorphism. 

The case when $\rho$ is of type $\infty_{+,-,0}$ is even easier.  
Let $A$ be a finite subset of $\mathbb{Q}$.  
Then the elements of $(A /\sim_{\gamma})$ can be mapped 
injectively to the set of orbitals of $\rho$ so that the parity is preserved.  
This follows from the fact that $\rho$ has infinitely many triples of orbitals of the type $+,-,0$. 

We now prove $1 \wedge 2 \rightarrow 3$. 
Assume that $\rho$ is not of type $\infty_{+,-,0}$. 
We now show that the assumption that $\gamma$ is of type $\infty_{+,-,0}$ leads to a contradiction. 
Indeed, let $k$ be the size of $\mathcal{L}_{\rho}$, i.e. the number of intervals in the canonical decomposition of 
$\mathcal{O}_{\rho}$. 
Choose a finite $A \subset \mathbb{Q}$ such that 
$(A\cup \gamma (A))/\sim_{\gamma}$ has an increasing sequence of size $k^k$ of triples of colored orbitals of parity $+,-,0$.  
Since $\gamma \in \mathcal{C}_{\rho}$ there is an automorphism $\sigma$ such that 
$\gamma \upharpoonright A = \sigma^{-1}\rho \sigma \upharpoonright A$.  
This means that $(\sigma (A)\cup \rho \sigma (A))/\sim_{\rho}$ has an increasing sequence of size $k^k$ of triples of colored orbitals of $\rho$ of parity $+,-,0$. 
We now see that there is an interval of the canonical decomposition of $(\mathcal{O}_{\rho}, \prec, \wp_{\rho} )$ 
which contains $k$ triples of orbitals of parity $+,-,0$.
Since $\chi$ does not have colors realizing this situation we  have a contradiction.  

Now assuming that $\gamma$ is not of type $\infty_{+,-,0}$ let $k = \mathsf{max} (|\mathcal{L}_{\gamma}|, |\mathcal{L}_{\rho}|)$. 
Find a finite $A \subset \mathbb{Q}$ such that for every $L\in \mathcal{L}_{\gamma}$ 
\begin{itemize} 
\item  the set $A\cap (\bigcup \{ O \, | \, O\in L\})$ is not empty,
\item if $L$ is of type $\infty_{\varepsilon_1 \varepsilon_2}$ or $\infty_{0}$ then $(A\cup \gamma (A))/\sim_{\gamma}$ has an increasing sequence of size $k^k$ of pairs of colored orbitals of parity $\varepsilon_1 , \varepsilon_2$ (resp. sequence of size $k^k$ of orbitals of parity $0$).  
\end{itemize}
Find an automorphism $\sigma$ such that 
$\gamma \upharpoonright A = \sigma^{-1}\rho \sigma \upharpoonright A$.  
Thus $\sigma$ maps $(A\cup \gamma (A))/\sim_{\gamma}$ to 
$(\sigma (A)\cup \rho \sigma (A))/\sim_{\rho}$. 
Furthermore, by the pigeon hole principle for every $L\in \mathcal{L}_{\gamma}$ of type  $\infty_{\varepsilon_1 \varepsilon_2}$ or $\infty_{0}$ the map $\sigma$ finds an increasing sequence of size $k$ of pairs of orbitals of parity $\varepsilon_1 , \varepsilon_2$ (resp. sequence of size $k$ of orbitals of parity $0$) in an interval of the canonical decomposition of $(\mathcal{O}_{\rho}, \prec, \wp_{\rho})$. 
It is clear that the latter one must be of type  $\infty_{\varepsilon_1 \varepsilon_2}$ (resp. $\infty_{0}$). 
If $L$ is of the type $\{ \varepsilon \}$ then $\sigma$ maps a  colored orbital of parity $\varepsilon$ to an interval of the canonical decomposition of $(\mathcal{O}_{\rho}, \prec, \wp_{\rho})$ which has a compatible type. 
Applying Definition \ref{d-epsilons} we summarize our argument as follows: when we find an interval of the canonical decomposition of $(\mathcal{O}_{\rho}, \prec, \wp_{\rho})$ which corresponds to an increasing sequence of size $k$ of pairs of colored orbitals of parity $\varepsilon_1 , \varepsilon_2$ (resp. sequence of size $k$ of orbitals of parity $0$), it can be arranged in the form of a $\chi$-homomorphism from the canonical decomposition $(\mathcal{L}_{\gamma}, \prec, \hat{\wp}_{\gamma})$ to the canonical decomposition $(\mathcal{L}_{\rho}, \prec, \hat{\wp}_{\rho})$.   
\end{proof}

\subsection{The case when $\rho$ is generic} 

When $\rho$ is generic in $\mathcal{C}_{\rho}$ condition 2 of Theorem \ref{equiva} can be formulated easier. 
In order to do this we need to know some basic properties of  such automorphisms $\rho$. 
It follows from Theorem \ref{equiva} that if 
$(\mathcal{O}_{\rho}, \prec_{\rho}, \wp_{\rho})$ is of type $\infty_{+,-,0}$ then 
$\mathcal{C}_{\rho} = \mathsf{Aut}(\mathbb{Q},<)$. 
In particular $\rho$ is a generic automorphism of $\mathsf{Aut} (\mathbb{Q}, <)$. 
It is well-known that in this case 
$(\mathcal{O}_{\rho}, \prec_{\rho})$ is a dense ordering such that for each parity $\varepsilon$ orbitals of this parity are dense in it, see \cite{truss94}. 

Let us consider the remaining cases. 
Let $\rho \in \mathsf{Aut}(\mathbb{Q},<)$, but $(\mathcal{O}_{\rho}, \prec_{\rho}, \wp_{\rho})$ be not of type $\infty_{+,-,0}$. 
According to Theorem \ref{decomposition} there is a decomposition 
\[ 
\mathcal{O}_{\rho} = L_1 \, \dot{\cup} \,  L_2 \, \dot{\cup} \ldots  \dot{\cup} \, L_n 
\] 
into finitely many intervals such that the corresponding $\chi$-colored ordering 
$(\mathcal{L}_{\rho},\preceq ,\hat{\wp}_{\rho})$ is canonical. 
Let us fix this $n$. 

Let $\mathcal{K}^{rch}_{\rho}$ be the family of all $\{ +,-,0\}$-colored linear orders $(A, <, \kappa )$  such that the following properties hold. 
\begin{quote} 
There is a $\kappa$-homomorphisms,  say $\phi$, from $(A, < , \kappa )$ to  
$(\mathcal{O}_{\rho},  \prec_{\rho}  ,\wp_{\rho})$ 
such that for any $L_i \in \mathcal{L}_{\rho}$ 
\end{quote} 
\begin{itemize} 
\item  the intersection $\phi (A)\cap L_i$ is not empty; 
\item if $L_i$ is of type $\{ + \}$ or $\{ - \}$, then $\phi (A)\cap L_i$ is single;  
\item if $L_i$ is of type $\infty_{\varepsilon_1 \varepsilon_2}$ or $\infty_{0}$, then it consists of an increasing sequence of size at least $n^n$ of pairs of the color $\varepsilon_1 , \varepsilon_2$ (resp. sequence of size at least $n^n$ of elements of color $0$).  
\end{itemize} 

The following proposition clears up the situation.

\begin{proposition} \label{gen-look}
Let $\rho \in \mathsf{Aut}(\mathbb{Q},<)$, but 
$(\mathcal{O}_{\rho}, \prec_{\rho}, \wp_{\rho})$ be not of type $\infty_{+,-,0}$. 
If $\gamma$ is a generic element of $\mathcal{C}_{\rho}$, then a decomposition of 
$(\mathcal{O}_{\gamma},\prec_{\gamma},  \wp_{\gamma} )$ corresponding to the canonical $\chi$-colored ordering defines the ordering which is $\wp$-isomorphic to $(\mathcal{L}_{\rho},\prec ,\hat{\wp}_{\rho})$.  

Furthermore, if $L$ is an interval of the decomposition above for $\gamma$ and $L$ is of type $\infty_{\varepsilon_1 \varepsilon_2}$ or $\infty_0$, then $L$ is a dense orderings and for each $\varepsilon \in \{ \varepsilon_1 ,\varepsilon_2 \}$ the orbitals of color $\varepsilon$ are dense in $L$. 
\end{proposition} 

\begin{proof} 
In the proof we use notation given before the formulation of the proposition (for example, $n$ and $\mathcal{K}^{rch}_{\rho}$). 
Let $\mathcal{D} \subseteq \mathcal{C}_{\rho}$ be the set of all $\alpha \in \mathcal{C}_{\rho}$ such that 
\begin{itemize} 
\item a decomposition of 
$(\mathcal{O}_{\alpha},\prec_{\alpha},  \wp_{\alpha} )$ corresponding the canonical $\chi$-colored ordering defines the ordering which is $\wp$-isomorphic to $(\mathcal{L}_{\rho},\prec ,\hat{\wp}_{\rho})$;   
\item if $L$ is an interval of the decomposition of the previous item and $L$ is of type $\infty_{\varepsilon_1 \varepsilon_2}$ or $\infty_0$, then $L$ is a dense orderings and for each $\varepsilon \in \{ \varepsilon_1 ,\varepsilon_2 \}$ the orbitals of color $\varepsilon$ are dense in $L$. 
\end{itemize} 
To prove the proposition we only need to show that the subset $\mathcal{D}$ is comeagre in $\mathcal{C}_{\rho}$.
To see the latter it suffices to note that for any finite $A\subseteq \mathbb{Q}$ the following set, say $\mathcal{D}_{A}$, is dense open in $\mathcal{C}_{\rho}$. 
\begin{quote} 
Let $\mathcal{D}_{A}$ consist of all $\alpha \in \mathcal{C}_{\rho}$ such that there is a finite $A' \subset \mathbb{Q}$ and $p'\in \mathcal{P}_{\rho}$ such that $A\subset A'$, 
$p'=\alpha \upharpoonright A'$, and
\end{quote}
\begin{itemize}
\item  associating to each $a'\in A'$ the color of the partial $p'$-orbital containing $a'$, the $\wp_{p'}$-colored ordering $A'$ becomes an element of $\mathcal{K}^{rch}_{\rho}$. 
\item whenever elements $a_1 <a_2$ from $A$ belong to a sequence from $A'$ of size at least $n^n$ of pairs of $\wp_{p'}$-colors $\varepsilon_1$ and $\varepsilon_2$ (resp. at least $n^n$ elements of color $0$) then there is a pair $c_1 < c_2$ from $A'$ of $\wp_{p'}$-colors $\varepsilon_1$ and $\varepsilon_2$ (resp.  $c$ of color $0$) between $a_1$ and $a_2$. 
\end{itemize} 
It is easy to see that $\mathcal{D}_{A}$ is open. 
To see that $\mathcal{D}_{A}$ is dense in $\mathcal{C}_{\rho}$ note that any $h\in \mathcal{P}_{\rho}$ embeds into a conjugate of $\rho$. 
In particular, having $h$ with $A \subseteq \mathsf{Dom} (h)$ 
we find an appropriate $A'$ such that $h$ extends to $\rho'\upharpoonright A'$ where $\rho'$ is a conjugate of $\rho$. 
On the other hand repeating the argument of the final part of Theorem \ref{equiva} it is easy to see that having $a_1 <a_2$ as in the second case of the definition of $\mathcal{D}_{A}$ the $\rho'$-orbitals of $a_1$ and $a_2$ must belong to an interval of type $\infty_{\varepsilon_1 \varepsilon_2}$ 
(resp. $\infty_0$) from a canonical decomposition of $(\mathcal{O}_{\rho'},\prec_{\rho'} )$.  
Thus $\rho'$ can be replaced by another conjugate (if necessary) where new orbitals between orbitals of $a_1$ and $a_2$ appear.  
\end{proof}

We now give a version of Theorem \ref{equiva} for $\mathcal{C}_{\rho}$ with $\rho$ generic in it.

\begin{proposition} \label{equival} 
Let $\rho \in \mathsf{Aut}(\mathbb{Q},<)$ and $\rho$ be generic in $\mathcal{C}_{\rho}$. 
Then for any $\alpha \in \mathsf{Aut}(\mathbb{Q},<)$ we have $\alpha \in \mathcal{C}_{\rho}$ if and only if there is a $\wp$-homomorphism from $(\mathcal{O}_{\alpha},\prec_{\alpha} ,\wp_{\alpha})$ to $(\mathcal{O}_{\rho},\prec_{\rho} ,\wp_{\rho})$. 
\end{proposition} 

\begin{proof} Sufficiency follows from  Theorem \ref{equiva}. 
Indeed, condition 2 of that theorem follows from existence of a $\wp$-homomorphism as in the formulation of the proposition.  

In order to prove necessity take $\alpha \in \mathcal{C}_{\rho}$ and present $\mathcal{O}_{\alpha}$ as a union of an increasing countable family of finite sets $\hat{\mathcal{O}}_n \subseteq \mathcal{O}_{\alpha}$, $n \in \omega$,  such that $\mathcal{O}_{\alpha}= \bigcup \{ \hat{\mathcal{O}}_n \, |\, n\in \omega \}$. 
If $\mathcal{O}_{\alpha}$ is finite, then we consider the family consisting of $\hat{\mathcal{O}}_1 = \mathcal{O}_{\alpha}$. 

Assume that $(\mathcal{O}_{\rho},\prec_{\rho} ,\wp_{\rho})$ is of type $\infty_{+,-,0}$. 
Then $\rho$ is a generic automorphism of $\mathsf{Aut} (\mathbb{Q}, <)$,  i.e.   
$(\mathcal{O}_{\rho}, \prec_{\rho})$ is a dense ordering such that for each parity $\varepsilon$ orbitals of this parity are dense in it, see \cite{truss94}. 
In this case an easy inductive argument gives a family of $\wp$-homomorphisms 
$\phi_i :  \hat{\mathcal{O}}_i \rightarrow \mathcal{O}_{\rho}$, $i\in \omega$, where each $\phi_{i+1}$ extends $\phi_i$.  
Then $\bigcup \phi_i$ gives the result. 

Let us consider the case when $(\mathcal{O}_{\rho},\prec_{\rho} ,\wp_{\rho})$ is not of type $\infty_{=,-,0}$. 
According to Theorem \ref{equiva} there is a $\chi$-homomorphism 
\[
\hat{\psi}: 
(\mathcal{L}_{\alpha},\prec ,\hat{\wp}_{\alpha})\rightarrow 
(\mathcal{L}_{\rho}, \prec, \hat{\wp}_{\rho}).  
\] 
In order to define a $\wp$-homomorphism 
$\phi_{i+1} :  \hat{\mathcal{O}}_{i+1} \rightarrow \mathcal{O}_{\rho}$, 
for every $L\in \mathcal{L}_{\alpha}$ let us define $\phi_{i+1}$ on  $L\cap \hat{\mathcal{O}}_{i+1}$.  
We assume that $\phi_i$ has been already defined. 

If $L$ is of type $\{ +  \} $, $\{ -  \}$ or $\{ 0 \}$ we define all $\phi_i$ by mapping each $L\cap \hat{\mathcal{O}}_{i}$ to the same element of $\hat{\psi}(L)$. 

If $L$ is of type $\infty_{\varepsilon_1 \varepsilon_2}$ or 
$\infty_0$ then step  $i+1$ looks as follows. 
Let $L\cap \hat{\mathcal{O}}_{i+1}= \{ C_1,\ldots ,C_k \}$ and let $L\cap \hat{\mathcal{O}}_{i}= \{ C_{i_1},\ldots ,C_{i_l} \}$. 
According to Proposition \ref{gen-look} for any $s\le l$ the interval $(\phi_i (C_{i_s}), \phi_i (C_{i_{s+1}}))$ is also of type $\infty_{\varepsilon_1 \varepsilon_2}$ (resp. $\infty_{0}$) in $\mathcal{O}_{\rho}$. 
Thus an extension of $\phi_{i}$ can be easily defined as a partial injective map on the interval $(C_{i_s}, C_{i_{s+1}})$. 
Using this method we build $\psi_{i+1}$ on $L\cap \hat{\mathcal{O}}_{i+1}$. 
Now take $\bigcup \phi_i$ as a required $\wp$-homomorphism $\phi$. 
\end{proof} 

\section{CAP for $\mathcal{P}_{\rho}$}

In this section we preserve the notation of Section 3. 
We remind the reader that $\mathcal{P}$ denotes the set of all finite partial isomorphisms of $(\mathbb{Q},<)$. 
When $\rho \in \mathsf{Aut}(\mathbb{Q},<)$ we denote by  
$\mathcal{C}_{\rho}$ the topological closure of $\rho^{\mathsf{Aut}(\mathbb{Q},<)}$ and 
\begin{quote} 
$\mathcal{P}_{\rho} = \{ p \in \mathcal{P} : p$ extends to an automorphism from $\mathcal{C}_{\rho} \}$. 
\end{quote} 
The main result of this section is Theorem \ref{gd+rch-cap} in Section 4.3. 
It states that the subspace $\mathcal{C}_{\rho}$ has a generic automorphism. 
The proof is based on some preparatory results concerning the following families. 
Let $\mathcal{P}^{+}$ be the family of all $p\in \mathcal{P}$ 
such that $(\mathsf{Dom} (p) \cup \mathsf{Rng} (p))/\sim_p$ 
consists of colored orbitals of parity $+$.  
The family $\mathcal{P}^{-}$ is defined in a similar way. 

\subsection{CAP for $\mathcal{P}^+$}

Let $p\in \mathcal{P}$ and 
$a,a' \in \mathsf{Dom} (p) \cup \mathsf{Rng} (p)$. 
We say that $(a,a')$ is a {\em bad pair of} $p$ 
(resp. in $\mathsf{Dom} (p) \cup \mathsf{Rng} (p)$) if one of the following properties holds: 
\begin{itemize}
\item $a<a'$ and $a,a'$ represent colored $p$-orbitals (which can coincide) such that they and all colored $p$-orbitals between them have parity $+$ and the values $p(a)$ and $p^{-1}(a')$ are not defined; 
\item $a<a'$ and $a,a'$ represent colored $p$-orbitals (which can coincide) such that they and all colored $p$-orbitals between them have parity $-$ and the values $p(a')$ and $p^{-1}(a)$ are not defined. 
\end{itemize} 
Let $\mathcal{P}^+_{gd}$ (resp. $\mathcal{P}^-_{gd}$) be the subset of $\mathcal{P}^+$ (resp. $\mathcal{P}^-$) 
consisting of all $p$ with a single colored $p$-orbital and without bad pairs. 
The following proposition is the main result of this subsection. 

\begin{proposition} \label{plus} 
The family $\mathcal{P}^+_{gd}$ (resp. $\mathcal{P}^-_{gd}$) is a  cofinal subfamily of $\mathcal{P}^+$ (resp. $\mathcal{P}^-$) and has the amalgamation property. 
\end{proposition}   

\begin{proof} 
We consider only the case of $\mathcal{P}^+$. 
In order to reach cofinality let $p \in \mathcal{P}^+$ and 
$\mathsf{Dom} (p) \cup \mathsf{Rng} (p) = A_1 \dot{\cup} \ldots \dot{\cup} A_{\ell}$ be a decomposition into colored orbitals (of parity $+$). 
Let $a'_i = \mathsf{min} A_i$ and $a''_i = \mathsf{max} A_i$, $i\le \ell$. 
Find $a_1, \ldots , a_{\ell}$ such that $a''_i < a_i < a'_{i+1}$, $i< \ell$, and extend $p$ to a map which takes each $a''_i$ to $a_i$ and each $a_i$ to $a'_{i+1}$, $i<\ell$.  
The extension which we get belongs to $\mathcal{P}^+$ and consists of a single orbital. 
Thus from now on we may assume that $p$ consists of a single colored orbital. 

It remains to eliminate bad pairs. 
Assume that $a<a'$ is a bad pair and $c_1 < c_2 < \ldots <c_k$ 
is a sequence from $\mathsf{Dom} (p) \cup \mathsf{Rng} (p)$ such that all $p^{-1}(c_i)$ belong to 
$\mathsf{Dom} (p) \cup \mathsf{Rng} (p)$ and 
\[ 
p^{-1} (c_1 )< a \le c_1 \, , \, p^{-1} (c_2 )\le c_1 < c_2 \, \ldots \, , \, p^{-1} (c_{k} )\le c_{k-1} < c_{k} \, , \, c_{k-1}\le a' \le c_k . 
\]
Since $p^{-1}(a')$ is not defined, $a'\not= c_k$. 
Choose a sequence of new elements $b_k < b_{k-1} < \ldots <b_1$ such that $p$ extends to a partial isomorphism with 
$p (b_1) = a'$, $b_k < a$ and $b_i = p^{-1}(b_{i-1} )$.  
This can be done by induction. 
For example we define $b_1$ as an element of $\mathbb{Q}$ such that for any $c \in \mathsf{Dom} (p) \cup \mathsf{Rng} (p)$ 
the inequality $b_1 <c$ holds only when there exists $c'\in \mathsf{Dom} (p) \cup \mathsf{Rng} (p)$ with 
$c'\le c$ and $a' < p(c')$. 
In particular $b_1 <c_{k-1}$. 
At the next step of induction the chosen $b_1$ will play the role of $a'$, etc. 
After such an extension of $p$ the number of bad paairs is reduced. 
It is now clear that $p$ can be extended to a finite partial  
isomorphism of $(\mathbb{Q}, <)$ which belongs to $\mathcal{P}^+_{gd}$.   

Let us prove AP for $\mathcal{P}^+_{gd}$. 
We will use some recipe from \cite{kuske-truss} and enrich it by a new trick. 
Assume that we have proper extensions $p_0 \subset p_1$ and $p_0 \subset p_2$ where $p_0, p_1 , p_2 \in \mathcal{P}^+_{gd}$. 
Let $A_i = \mathsf{Dom} (p_i) \cup \mathsf{Rng} (p_i)$, $i\le 2$. 
We may assume that $A_0 = A_1 \cap A_2$. 

Let us start with the (simplest) case when for each $i\in \{ 1,2\}$ each partial $p_i$-orbit in $A_i$ is contained in $A_0$ or has empty intersection with $A_0$.  
In this case $p_1 \cup p_2$ is a partial isomorphism on $A_1 \cup A_2$. 
In order to make it to belong to $\mathcal{P}^+_{gd}$ we apply the procedure which was used above for cofinality.   

Let us consider the case when there are partial $p_0$-orbits in $A_0$ which have proper extensions to $p_1$-orbits in $A_1$ (resp. $p_2$-orbits in $A_2$).  
Apply induction by $|A_1 \setminus A_0 |+|A_2 \setminus A_0|$.  
It is clear that: 
\begin{itemize} 
\item there is $a\in A_0$ such that $p_0 (a)$ is not defined but one of $p_1 (a)$ or  $p_2 (a)$ is defined {\bf or}  
\item 
there is $a\in A_0$ such that $p^{-1}_0 (a)$ is not defined but one of $p^{-1}_1 (a)$ or $p^{-1}_2 (a)$ 
is defined. 
\end{itemize}  
We consider the first case, the second one is similar. 
Then neither $p_1 (a)$ nor $p_2 (a)$ belongs to $A_0$ since otherwise they form bad pairs with $a$ in $A_0$. 

Assume that both $p_1 (a)$ and  $p_2 (a)$ are defined and $p_1 (a)\le p_2 (a)$ (the case $p_1 (a)\ge p_2 (a)$ is symmetric). 
Then for any interval $(c_1 ,c_2 )$ with $c_1, c_2 \in A_0$, 
\[ 
c_1 \le p_1 (a) \le c_2 \Leftrightarrow c_1 \le p_2 (a) \le c_2.  
\]
Indeed, if for example $p_1 (a) \le c_2 < p_2 (a)$ then there is $c'\in A_0$ such that $p(c') = c_2$ (the pair $a<c_2$ is not bad in $A_0$). 
Since  $p_1 (a) \le c_2$ we have $a \le c'$ and $p_2 (a) \le p_2 (c') = c_2$, a contradiction.  

We now make some partial amalgamation as follows. 
Let $p'_0$ be the extension of $p_0$ by the assignment $a \rightarrow p_1 (a)$. 
Note that $p'_0$ does not have bad pairs. 
Let $p'_2$ be obtained from $p_2$ by replacing the pair (of its graph) $a \rightarrow p_2 (a)$ by $a\rightarrow p_1 (a)$ and if $p_2 (a)$ is mapped by $p_2$ to some $d\in A_2$ we replace $p_2 (a) \rightarrow d$ by $p_1 (a)\rightarrow d$. 
Doing so we should similarly correct other $p_2$-orbits in $A_2$ which intersect $(a,p_2(a))$: when $c \in (a,p_2(a)) \cap A_2$ represents such an orbit we just replace it in this orbit by some $c' \in (a,p_1(a))$.    
As a result $p'_2$ is isomorphic with $p_2$ over $p_0$. 
If $A'_i = \mathsf{Dom} (p'_i) \cup \mathsf{Rng} (p'_i)$, $i \in \{ 0,2\}$, then 
\[
|A_1 \setminus A'_0 |+|A'_2 \setminus A'_0| < |A_1 \setminus A_0 |+|A_2 \setminus A_0|,  
\] 
i.e. we have made a step to full amalgamation. 

The case when $p_1 (a)$ is defined but $p_2 (a)$ is not, can be reduced to the previous one. 
Let $q\in \mathbb{Q}\setminus A_2$ be chosen as a realization (for $z$) of  
\[ 
\forall x\in A_2 (x< a \le p_2 (x) \rightarrow p_2 (x) < z )  
\]
which is less than all realizations of this formula from $A_2$.  
We extend $p_2$ by the assignment $a \rightarrow q$. 
It is clear that the obtained $p'_2$ does not have bad pairs, so the argument of the previous case can be applied.   
\end{proof}

According to Theorem \ref{equiva} for any $\rho$ having only positive orbitals the class $\mathcal{C}_{\rho}$ consists of all automorphisms $\alpha$ which have only positive orbitals too. 
Let us call it $\mathcal{C}^{+}$. 
According to Proposition \ref{plus} (above) and  Theorem 1.2 of \cite{iva99} this class has a generic automorphism. 
It is curios that up to conjugacy this is the standard shift: $st (x) = x+1$. 
Indeed, if $\gamma \in \mathcal{C}^+$ and has at least two orbitals then fix two representatives of them, say $a_1$ and $a_2$. 
It is an easy exercise that the set of all automorphisms from $\mathcal{C}^+$ with $a_1$ and $a_2$ in the same orbital, is dense and open in $\mathcal{C}^+$.  
Thus $\gamma$ cannot be generic in $\mathcal{C}^+$. 

\subsection{Gluing orbitals} 

Let $\mathcal{P}_{gd}$ be the set of all partial isomorphisms 
$p\in\mathcal{P}$ with the following properties: 
\begin{itemize} 
\item there are no bad pairs in colored $p$-orbitals of $(\mathsf{Dom} (p) \cup \mathsf{Rng} (p))/\sim_p$ of parities $\pm$;
\item if colored $p$-orbitals $O$ and $O'$ of parities $\pm$ are neighbors in the natural ordering of 
$(\mathsf{Dom} (p) \cup \mathsf{Rng} (p))/\sim_p$ then 
they have different parities.  
\end{itemize}  
We will call the members of $\mathcal{P}_{gd}$ {\em good isomorphisms}. 

\begin{remark}  \label{emb}   
Assume that $p\in \mathcal{P}_{gd}$ and 
$A = \mathsf{Dom} (p) \cup \mathsf{Rng} (p)$.    
Then any $\wp$-homomorphism from $(A/\sim_{p}, \prec_p, \wp_p )$ to any $\wp$-colored linear ordering is injective. 
{ \em Indeed, according to Definition \ref{kappa-hom} if $\phi$ is a such a homomorphism and $\phi (Q) = \phi (Q')$, then the parity of $Q$ and $Q'$ is the same and not equal to $0$.  
Such orbitals should be separated by an orbital of a different color, which is impossible. 
}
\end{remark}

\begin{lemma} \label{cof}
The family $\mathcal{P}_{gd}$ is cofinal in $\mathcal{P}$.
Furthermore, assume that $\gamma\in \mathsf{Aut}(\mathbb{Q},<)$ . 
Then the family $\mathcal{P}_{gd} \cap\mathcal{P}_{\gamma}$ is cofinal in $\mathcal{P}_{\gamma}$. 
\end{lemma}
 
We will deduce this lemma from the following procedure. 

\bigskip 
\noindent
$\mathsf{Gluing} \, \mathsf{orbitals}$: 

Let $p \in \mathcal{P}$. 
Consider $(\mathsf{Dom} (p) \cup \mathsf{Rng} (p))/\sim_p$ as a linearly ordered set induced by $(\mathbb{Q},<)$. 
Assume that $O_i, O_{i+1} ,\ldots, O_j$ is a maximal interval of the same color ($+$ or $-$).  
Their elements belong to the open interval of $\mathbb{Q}$ determined by the last element of $O_{i-1}$ and the first one of $O_{j+1}$. 
Applying Proposition \ref{plus} inside this interval we obtain an extension of $p$ where $O_i, O_{i+1} ,\ldots, O_j$ become subsets of a single orbital without bad pairs. 
Repeating this procedure for each maximal interval of the same parity we eventually obtain an element of $\mathcal{P}_{gd}$. 
   
\bigskip 
   
\noindent
{\em Proof of Lemma \ref{cof}.} 
The first sentence of the lemma follows directly from   
$\mathsf{Gluing} \, \mathsf{orbitals}$.  
If $p \in \mathcal{P}_{\gamma}$ then applying this procedure we obtain $p_1 \in \mathcal{P}_{gd}$ extending $p$.   
Note that the ordered set $((\mathsf{Dom} (p_1 ) \cup \mathsf{Rng} (p_1)) /\sim_{p_1}, \prec_{p_1}, \wp_{p_1})$ embeds into $((\mathsf{Dom} (p) \cup \mathsf{Rng} (p))/\sim_p, \prec_p , \wp_p)$ by a $\wp$-homomorphism. 
If $\gamma'$ is a conjugate of $\gamma$ which extends $p$ then there is a finite restriction of $\gamma'$ which is of the type of $p_1$. 
By ultrahomogeneity of  $(\mathbb{Q},<)$ the partial isomorphism $p_1$ can be extended to a conjugate of $\gamma$.  
$\Box$

\subsection{CAP for $\mathcal{P}_{\rho}$} 

Let $\rho \in \mathsf{Aut}(\mathbb{Q},<)$. 
Consider $(\mathcal{O}_{\rho}, \prec_{\rho} , \wp_{\rho} )$
where $\wp_{\rho}$ is the corresponding parity function. 
There are two cases. 
If $(\mathcal{O}_{\rho}, \prec_{\rho} , \wp_{\rho} )$ is of type $\infty_{+,-,0}$ then according to Theorem \ref{equiva}, 
$\mathcal{C}_{\rho} = \mathsf{Aut}(\mathbb{Q},<)$. 
In particular we have $\mathcal{P}_{\rho} = \mathcal{P}$. 
It is well-known that $\mathcal{P}$ satisfies CAP (see Theorem 1.1 in \cite{kuske-truss}), i.e. $\mathcal{C}_{\rho}$ has a generic automorphism (and its description is well-known, \cite{truss94}).  
Note that by Lemma \ref{cof} the family $\mathcal{P}_{gd}$ is cofinal in $\mathcal{P}$. 
According to the proof of Theorem 1.1 of \cite{kuske-truss} $\mathcal{P}_{gd}$ is contained in a subfamily of $\mathcal{P}$ with AP. 
Thus $\mathcal{P}_{gd}$ has AP too. 
 
Now assume that $\mathcal{O}_{\rho}$ is not of type $\infty_{ +,-,0}$. 
Applying Theorem  \ref{decomposition} we find a $\chi$-colored ordering $(\mathcal{L}_{\rho}\prec_{\rho}, \hat{\wp}_{\rho})$ corresponding to a canonical decomposition of $(\mathcal{O}_{\rho}, \prec_{\rho},\wp_{\rho})$ into finitely many intervals: 
\[ 
\mathcal{L}_{\rho} =( L_1 \, < \,  L_2 \, < \ldots  < \, L_n ) \, 
\mbox{ where } \, 
\mathcal{O}_{\rho} = L_1 \, \dot{\cup} \,  L_2 \, \dot{\cup} \ldots  \dot{\cup} \, L_n .
\] 
Let $\mathcal{K}_{\rho}$ be the age of $(\mathcal{O}_{\rho}, \prec_{\rho},\wp_{\rho})$. 
As above we consider $\mathcal{K}_{\rho}$ together with the family of all $\wp$-homomorphisms. 
Note that when $p$ is a partial automorphism from $\mathcal{P}_{\rho}$ and $\sim_p$ is the corresponding partial orbitals equivalence relation then the finite structure $((\mathsf{Dom} (p) \cup \mathsf{Rng} (p)) /\sim_{p}, \prec_{p}, \wp_{p})$
belongs to $\mathcal{K}_{\rho}$, i.e. is mapped into $(\mathcal{O}_{\rho}, \prec_{\rho},\wp_{\rho})$ by a $\wp$-homomorphism. 
Having such a homomorphism we obtain a decomposition of  $(\mathsf{Dom} (p) \cup \mathsf{Rng} (p)) /\sim_{p}$ induced by the canonical decomposition of $\mathcal{L}_{\rho}$. 

Let $\mathcal{K}^{rch}_{\rho}$ be the subfamily of $\mathcal{K}_{\rho}$ defined according to Definition \ref{rch} with respect to the decomposition above. 

\begin{definition} 
Let $p$ be a partial automorphism from $\mathcal{P}_{\rho}$ and $\sim_p$ be the corresponding partial orbitals equivalence relation. 
We say that $p$ is rich if the finite structure $((\mathsf{Dom} (p) \cup \mathsf{Rng} (p)) /\sim_{p}, \prec_{p}, \wp_{p})$
belongs to $\mathcal{K}^{rch}_{\rho}$. 

Let $\mathcal{P}^{rch}_{\rho}$ be the family of all rich isomorphisms from $\mathcal{P}_{\rho}$. 
\end{definition} 

The following theorem demonstrates that the family of all partial isomorphisms which are good and rich witnesses CAP. 

\begin{theorem} \label{gd+rch-cap}
Let $\rho \in \mathsf{Aut}(\mathbb{Q},<)$.  
Then the subspace $\mathcal{C}_{\rho}$ has a generic automorphism. 
Furthermore, if $\mathcal{O}_{\rho}$ is not of type $\infty_{ +,-,0}$ then the family $\mathcal{P}^{rch}_{\rho} \cap \mathcal{P}_{gd}$ is a cofinal subfamily of $\mathcal{P}_{\rho}$ and, moreover, it has the amalgamation property.  
\end{theorem} 

\begin{proof} 
As we have already observed above it suffices to consider the case when $\mathcal{O}_{\rho}$ is not of type $\infty_{ +,-,0}$. 
In order to prove cofinality of $\mathcal{P}^{rch}_{\rho} \cap \mathcal{P}_{gd}$ take any $p\in \mathcal{P}_{\rho}$. 
Let $A$ be a finite subset of $\mathbb{Q}$ such that  
$(A /\sim_{p_1}, \prec_{p_1}, \wp_{\rho} \upharpoonright A )$
belongs to $\mathcal{K}^{rch}_{\rho}$, where $p_1 = \rho \upharpoonright A$.
Using JEP for $\mathcal{P}_{\rho}$ we may assume that $p$ extends $p_1$ for some $A$. 
This induces a $\wp$-homomorphism, say $\phi_1$, from 
$(A /\sim_{p_1}, \prec_{p_1}, \wp_{\rho} \upharpoonright A )$ to $( (\mathsf{Dom} (p) \cup \mathsf{Rng} (p)) /\sim_{p}, \prec_{p}, \wp_{p} )$. 
Let us apply procedure 
$\mathsf{Gluing} \, \mathsf{orbitals}$ to $p$. 
If $p_2$ is the result of this procedure, then there is an obvious $\wp$-homomorphism, say $\phi_2$, from 
$((\mathsf{Dom} (p) \cup \mathsf{Rng} (p)) /\sim_{p}, \prec_{p}, \wp_{p} )$ onto $( (\mathsf{Dom} (p_2) \cup \mathsf{Rng} (p_2)) /\sim_{p_2}, \prec_{p_2}, \wp_{p_2} )$. 
As in the proof of Lemma \ref{cof} we have $p_2 \in \mathcal{P}_{\rho}$. 
Thus $\phi_2 \circ \phi_1$ is a $\wp$-homomorphism from 
$(A /\sim_{p_1}, \prec_{p_1}, \wp_{\rho} \upharpoonright A )$ to $( (\mathsf{Dom} (p_2) \cup \mathsf{Rng} (p_2)) /\sim_{p_2}, \prec_{p_2}, \wp_{p_2} )$. 
Now verifying Definition \ref{rch} for the latter ordering one sees that the only condition which can fail is the last one in the case when $L_i$ of type $\infty_{\varepsilon_1 \varepsilon_2}$. 
It may happen that 
$((\mathsf{Dom} (p_2) \cup \mathsf{Rng} (p_2)) /\sim_{p_2} )\cap L_i$ starts with $\varepsilon_2$ or ends by $\varepsilon_1$. 
We can remedy this situation at the moment when we apply JEP for $p$ and $p_1$. 
At this stage we can slightly extend the result of that amalgamation by new partial orbitals. 

Let us prove the amalgamation property for $\mathcal{P}^{rch}_{\rho} \cap \mathcal{P}_{gd}$. 
Take a pair of extensions $p_0 \subseteq p_1$ and $p_0 \subseteq p_2$ of elements of $\mathcal{P}^{rch}_{\rho} \cap \mathcal{P}_{gd}$. 
They induce natural $\wp$-homomorphisms, say $\phi_1$ and $\phi_2$, among $( (\mathsf{Dom} (p_i) \cup \mathsf{Rng} (p_i)) /\sim_{p_i}, \prec_{p_i}, \wp_{p_i} )$, $i\le 2$. 
According to Remark \ref{emb} they are injective and according to the claim from the proof of Theorem  \ref{rch-amal} 
they preserve the intervals $L_i$ of the canonical decomposition 
$(\mathcal{L}_{\rho},\preceq ,\hat{\wp}_{\rho})$.
Applying Theorem \ref{rch-amal} we find a $\{ +,-,0\}$-colored linear ordering $(L, <, \kappa ) \in \mathcal{K}^{rch}_{\rho}$ and embeddings $\psi_i$ of $( (\mathsf{Dom} (p_i) \cup \mathsf{Rng} (p_i)) /\sim_{p_i}, \prec_{p_i}, \wp_{p_i} )$ into it ($1\le i \le 2$) which give the corresponding amalgamation. 
Note that this amalgamation is disjoint: it is explicitly stated in the proof of Theorem \ref{rch-amal}.  
Furthermore, we may assume that 
\[ 
L = \psi_1 ( (\mathsf{Dom} (p_1) \cup \mathsf{Rng} (p_1)) /\sim_{p_1}) \cup \psi_2 ((\mathsf{Dom} (p_2) \cup \mathsf{Rng} (p_2)) /\sim_{p_2}). 
\] 
Let us build a partial isomorphism $p_3 \in \mathcal{P}^{rch}_{\rho} \cap \mathcal{P}_{gd}$ such that $p_3$ is an amalgamation of $p_1$ and $p_2$ over $p_0$ and $( (\mathsf{Dom} (p_3) \cup \mathsf{Rng} (p_3)) /\sim_{p_3}, \prec_{p_3}, \wp_{p_3} ) = (L, <, \kappa )$. 

Let $L_i \in \mathcal{L}_{\rho}$. 
Define $A_i = \bigcup (((\mathsf{Dom} (p_0) \cup \mathsf{Rng} (p_0))/\sim_{p_0} )\cap L_i )$, i.e. the union of the corresponding partial orbitals. 
We fix rational numbers $q_i$ such that 
\[
A_1 < q_1 < A_2 < q_2 < A_3 < \ldots < q_{n-1} <A_n . 
\]
Using ultrahomogeneity of $(\mathbb{Q},<)$ we can squeeze the set  
\[ 
\bigcup  ((\mathsf{Dom} (p_1) \cup \mathsf{Rng} (p_1))/\sim_{p_1})\cap L_i ) \cup  ((\mathsf{Dom} (p_2) \cup \mathsf{Rng} (p_2)) /\sim_{p_2} ) \cap L_i )
\] 
into $(q_{i-1}, q_i )$ (resp. $(-\infty, q_1)$ or $(q_{n-1}, \infty)$) over $A_i$. 
If $|\mathcal{L}_{\rho}|=1$, there is no need to do this. 

Let $\iota \in L \cap L_i$. 
We start with the case when $\iota$ is represented both by $\psi_1$ and by $\psi_2$. 
Then there is $O\in (\mathsf{Dom} (p_0) \cup \mathsf{Rng} (p_0)) /\sim_{p_0}$ such that 
$\psi_1 \circ \phi_1 (O) = \iota = \psi_2 \circ \phi_2 (O )$. 
The case when $O$ is of parity $0$ is obvious: let $p_3$ fix the element of $O$. 
When $O$ is of parity $+$ or $-$ we apply Proposition \ref{plus} as follows.  
Let $O_1$ and $O_2$ be extensions of $O$ which are colored orbitals of $p_1$ and $p_2$ respectively. 
According to Proposition \ref{plus} $p_1 \upharpoonright O_1$ and $p_2 \upharpoonright O_2$ can be amalgamated over $p_0 \upharpoonright O$ into some partial isomorphism, say $p'$, defined on a single orbital, say $O_3$.   
This $p'$ is going to be  the restriction of $p_3$ to the colored orbital corresponding to $\iota$. 
We can again squeeze $O_3$ into any open interval containing $O$. 
Thus we can organize amalgamation in the parallel way for all $\iota \in L$ which are represented by both $\psi_1$ and $\psi_2$. 

Note that this procedure amalgaamates the $L_i$-parts of $p_1$ and $p_2$ over $p_0$ for all $L_i$ which are single elements (i.e. of $\chi$-colors $\{ + \}, \{ - \}, \{ 0 \}$).   
Just in case we remind the reader that these parts are not empty. 

Assume that $\iota$ is represented by $\psi_1$ but is not represented by $\psi_2$. 
Then the action of $p_3$ on the orbital $\psi^{-1}_1 (\iota )$ is defined to be the action of $p_1$ on this set.  
In the opposite case we similarly define $p_3$ on $\psi^{-1}_2 (\iota )$. 
We place these orbitals in the intervals of $(\mathbb{Q},<)$ which naturally arise after the amalgamation described in the previous paragraph, i.e. between consecutive orbitals corresponding to elements $\iota \in L$ represented by both $\psi_1$ and $\psi_2$.  
Doing so we place the new orbitals of $p_3$ which we define at this stage according to the ordering of $(L,<)$. 
Since $(L,< \kappa) \in \mathcal{K}^{rch}_{\rho}$, we easily see that  $p_3 \in \mathcal{P}^{rch}_{\rho} \cap \mathcal{P}_{gd}$.  

Now the main statement of our theorem follows from Theorem 1.2 of \cite{iva99} (see Section 1.1 above). 
\end{proof} 

\section{Partial ordering $B_n$ and the betweenness relation} 

In this section we show how Theorem \ref{gd+rch-cap} can be applied to other highly homogeneous structures and to ultrahomogeneous orderings discussed in Section 1.3. 
In Section 5.1 we consider the case of $B_n$ for finite $n$. 
It is summarized in Theorem \ref{B_n}.  
In Section 5.2 we apply this analysis to the highly homogeneous structure $(\mathbb{Q}, B)$. 
In Section 5.3 we discuss the remained cases.

\subsection{Ordering $B_n$} 
Let $n\in \omega$. 
Consider $B_n = [n]\times \mathbb{Q}$,  
the ultrahomogeneous partially ordered set with respect to the ordering 
\[ 
(a,q) < (b,q') \, \Leftrightarrow \, a = b \wedge q <q'. 
\]
Any $\rho \in \mathsf{Aut}(B_n )$ induces a permutation from $S_n$. 
This is the projection of $\rho$ on $[n]$. 
We denote it by $\sigma_{\rho}$. 
Let $\iota_{\rho}$ be the order of $\sigma_{\rho}$. 

Let $\bar{a}$ be an orbit of $\sigma_{\rho}$ on $[n]$ (with elements $a_i$) and $|\bar{a}|$ be its size. 
We assume that $\sigma_{\rho}$ takes $a_i$ to $a_{i+1}$ (and $a_{|\bar{a}|}$ to $a_1$). 
Then $\rho^{|\bar{a}|}$ is an automorphism of each copy $(a_i, \mathbb{Q})$. 
Furthermore, $\rho$ is an isomorphism from $((a_{i}, \mathbb{Q}), < )$ to $((a_{i+1},\mathbb{Q}), < )$ which preserves the graph of $\rho^{|\bar{a}|}$. 
Let $(\mathcal{O}_{\rho^{|\bar{a}|}}, \prec_{\rho^{|\bar{a}|}})$ be the ordering of orbitals of $\rho^{|\bar{a}|}$ on $((a_{1},\mathbb{Q}), < )$. 
Then $\rho^i$ defines an isomorphism from $(\mathcal{O}_{\rho^{|\bar{a}|}}, \prec_{\rho^{|\bar{a}|}}, \wp_{\rho^{|\bar{a}|}})$ to the corresponding ordering of orbitals of $((a_{i},\mathbb{Q}), < )$. 

To analyze $\mathcal{P}_{\rho}$ we start with $\mathcal{P}_{\rho^{|\bar{a}|}}$ for the structure $((a_1, \mathbb{Q}),<)$. 
The approach of Sections 3 and 4 can be applied in this case. 
In particular $\mathcal{P}_{gd} \cap \mathcal{P}_{\rho^{|\bar{a}|}}$ can be considered and in the case when $\rho^{|\bar{a}|}$ is not of type $\infty_{+,-,0}$ on $(a_1, \mathbb{Q})$ a canonical decomposition of $(\mathcal{O}_{\rho^{|\bar{a}|}}, \prec_{\rho^{|\bar{a}|}}, \wp_{\rho^{|\bar{a}|}})$ and the family $\mathcal{P}^{rch}_{\rho^{|\bar{a}|}}$ can be defined.  
We now describe how these families can be considered in $\mathcal{P}_{\rho}$. 
  
We say that $p\in \mathcal{P}_{\rho}$ is {\em symmetric} if for every orbit $\bar{a}$ of $\sigma_{\rho}$ the map $p^i$, $i\le |\bar{a}|$, defines an isomorphism from 
$((\mathsf{Dom} (p) \cup \mathsf{Rng} (p)) \cap (a_1,\mathbb{Q}),<)$ to 
$((\mathsf{Dom} (p) \cup \mathsf{Rng} (p))\cap (a_{i},\mathbb{Q}),<)$ which preserves the partial action of $p^{|\bar{a}|}$ on these ordered sets. 
Note that this condition implies that 
\[ 
\mathsf{Rng} (p) \cap (a_{1},\mathbb{Q}) \subseteq \mathsf{Dom} (p)\cap (a_{1},\mathbb{Q}) \mbox{ and }
\mathsf{Dom} (p)\cap (a_{|\bar{a}|},\mathbb{Q}) \subseteq \mathsf{Rng} (p) \cap (a_{|\bar{a}|},\mathbb{Q}). 
\] 
Furthermore, 
\[ 
\mathsf{Dom} (p)\cap (a_{i},\mathbb{Q}) = \mathsf{Rng} (p) \cap (a_{i},\mathbb{Q}) \, \mbox{, for }  \, 1 <i<|\bar{a}|. 
\] 
When $\sigma_{\rho}$ is trivial, any $p\in \mathcal{P}_{\rho}$ is symmetric. 
We denote by $\mathcal{P}^{sym}_{\rho}$ the subfamily of $\mathcal{P}_{\rho}$ of all symmetric partial isomorphisms.

\begin{lemma} \label{b_n}
Let 
$\rho \in \mathsf{Aut} (B_n)$ with non-trivial $\sigma_{\rho}$. 
Then the family $\mathcal{P}^{sym}_{\rho}$ is cofinal in $\mathcal{P}_{\rho}$. 
\end{lemma} 

\begin{proof} 
Let $p \in \mathcal{P}_{\rho}$. 
For simplicity we consider the case when the whole $[n]$ is the orbit of $\sigma_{\rho}$ with $\sigma_{\rho}(i) = i+1$.   
Thus we view $p$ as a collection of partial isomorphisms:  
\[ 
((\mathsf{Dom} (p) \cup \mathsf{Rng} (p)) \cap (a_{i},\mathbb{Q}) \rightarrow 
((\mathsf{Dom} (p) \cup \mathsf{Rng} (p))\cap (a_{i+1} ,\mathbb{Q}) \, , \, i<\iota_{\rho} ,  
\] 
and we view each $((\mathsf{Dom} (p) \cup \mathsf{Rng} (p)) \cap (a_{i},\mathbb{Q})$ as the support of the partial isomorphisms defined by $p^{\iota_{\rho}}$. 

The elements of $\mathsf{Dom} (p)\cap (a_{1}, \mathbb{Q})$
and $\mathsf{Rng} (p)\cap (a_{2}, \mathbb{Q})$ 
give two decompositions of $(a_{1},\mathbb{Q})$ and $(a_{2},\mathbb{Q})$, say $q_1 < q_2 < \ldots < q_s$ and 
$p (q_1 ) < p(q_2 ) < \ldots < p(q_s )$ respectively. 
If there is $q \in  \mathsf{Rng} (p) \cap (a_{1},\mathbb{Q})$ with undefined $p(q)$ find $(q_i, q_{i+1})$ (or the half-lines of $q_1$ or $q_s$) which contains this $q$. 
Then choosing some $q' \in (p(q_i ), p(q_{i+1}))$ we extend $p$ by defining $p(q) = q'$. 
If $q$ is a $p^{\iota_{\rho}}$-image of some $q_j$ then according to this definition $p(q)$ becomes the $p^{\iota_{\rho}}$-image of $p(q_j)$, i.e. the extension preserves $p^{\iota_{\rho}}$. 
This reduces the number of points of $(a_1 ,\mathbb{Q})$ for which $p$ is undefined.   
If there is $q'' \in  \mathsf{Dom} (p) \cap  (a_{2},\mathbb{Q})$ without $p$-preimages we find a preimage by the same method. 

At the second stage of the procedure we repeat this construction for the new version of $\mathsf{Dom} (p)\cap (a_{2}, \mathbb{Q})$ and for $\mathsf{Rng} (p)\cap (a_{3}, \mathbb{Q})$ (with appropriate correction of $\mathsf{Rng} (p)\cap (a_{1}, \mathbb{Q})$) and so forth till $i = \iota_{\rho}-1$. 
After the last stage we obtain an extension of $p$ which is a symmetric partial isomorphism. 

In the case when $[n]$ consists of several orbits we should apply the construction above to each orbit of $[n]$ 
\end{proof} 

Given $p \in \mathcal{P}_{\rho}^{sym}$ we associate to each $\sigma_{\rho}$-orbit $\bar{a}$ the action of $p^{|\bar{a}|}$ on 
$(\mathsf{Dom} (p) \cup \mathsf{Rng} (p)) \cap (a_{1},\mathbb{Q})$.   
This data together with isomorphisms $p^i$ from the definition, determines $p$.  
It is worth noting here that 
\[ 
\mathsf{Rng} (p) \cap (a_{1},\mathbb{Q}) =  \mathsf{Rng} (p^{|\bar{a}|}) \cap (a_{1},\mathbb{Q}). 
\] 
It is reasonable to say that such $p$ belongs to   
$\mathcal{P}_{gd} \cap \mathcal{P}_{\rho}$ (or to $\mathcal{P}^{rch}_{\rho}$) if for each $\sigma_{\rho}$-orbit $\bar{a}$ we additionally have that 
$p^{|\bar{a}|}$ considered in $(a_1, \mathbb{Q})$ belongs to $\mathcal{P}_{gd}$ (resp. $\mathcal{P}^{rch}_{\rho^{|\bar{a}|}}$). 

\begin{theorem} \label{B_n}
Let 
$\rho \in \mathsf{Aut} (B_n)$. 
Then $\mathcal{C}_{\rho}$ contains a generic automorphism. 
Furthermore, the family $\mathcal{P}^{sym}_{\rho}$ contains a cofinal subfamily of $\mathcal{P}_{\rho}$ with the amalgamation property. 
\end{theorem} 

\begin{proof}
For each orbit $\bar{a}$ of $\sigma_{\rho}$ consider 
$\mathcal{P}_{\rho^{|\bar{a}|}}$ with respect to the action on $(a_{1},\mathbb{Q})$.  
By Theorem \ref{gd+rch-cap} this family has a cofinal subfamily with the amalgamation property. 
Let us denote it by $\mathcal{P}^{cap}_{\rho^{|\bar{a}|}}$. 
By $\mathcal{P}^{cap}_{\rho}$ we denote the family of all 
$p \in \mathcal{P}^{sym}_{\rho}$ such that for each $\sigma_{\rho}$-orbit $\bar{a}$ 
the action of $p^{|\bar{a}|}$ on 
$(\mathsf{Dom} (p) \cup \mathsf{Rng} (p)) \cap (a_{1},\mathbb{Q})$ belongs to $\mathcal{P}^{cap}_{\rho^{|\bar{a}|}}$.   

\noindent 
{\em Claim 1.} 
The family $\mathcal{P}^{cap}_{\rho}$ is cofinal in $\mathcal{P}^{sym}_{\rho}$. \\ 
Take any $p \in \mathcal{P}^{sym}_{\rho}$ and any $\sigma_{\rho}$-orbit $\bar{a}$. 
Extend the action of $p^{|\bar{a}|}$ on 
$(\mathsf{Dom} (p) \cup \mathsf{Rng} (p)) \cap (a_{1},\mathbb{Q})$ to some $\mathfrak{p}_{\bar{a}}\in \mathcal{P}^{cap}_{\rho^{|\bar{a}|}}$ defined on $(a_1, \mathbb{Q})$. 
Now consequently extend each $p^i$, $0<i<|\bar{a}|$, 
to an isomorphism, say $\hat{p}^i$, from $\mathsf{Dom} (\mathfrak{p}_{\bar{a}}) \cup \mathsf{Rng} (\mathfrak{p}_{\bar{a}})$ into $(a_{i},\mathbb{Q})$. 
Finally extend the the map defined by $p$ 
\[ 
p^{|\bar{a}|-1}(\mathsf{Dom} (p) \cap (a_{1},\mathbb{Q})) \rightarrow \mathsf{Rng} (p) \cap (a_{1},\mathbb{Q})   
\] 
to the map 
\[ 
\hat{p}^{|\bar{a}|-1}(\mathsf{Dom} (\mathfrak{p}_{\bar{a}})) \cap (a_{|{\bar{a}|}},\mathbb{Q}) \rightarrow \mathsf{Rng} (\mathfrak{p}_{\bar{a}})    
\]  
which takes each $\hat{p}^{|\bar{a}|-1} (v)$ with  
$v\in \mathsf{Dom} (\mathfrak{p}_{\bar{a}})$ to $\mathfrak{p}_{\bar{a}}(v)$.  
It is easy to see that $\hat{p}$ is an $\bar{a}$-restriction of some member of $\mathcal{P}^{cap}_{\rho}$. 
Applying this construction to each orbit of $\sigma_{\rho}$
we get the extension realizing the statement of the claim. 

\noindent 
{\em Claim 2.} The family $\mathcal{P}^{cap}_{\rho}$ has the amalgamation property. 

Before the proof of this claim let us consider the situation  when $p_0, p_1 \in \mathcal{P}^{cap}_{\rho}$ and $p_0 \subseteq p_1$. 
Then obviously for $\bar{a}$ as above 
\[ 
(\mathsf{Dom} (p_0) \cup \mathsf{Rng} (p_0)) \cap (a_{1},\mathbb{Q}) \subseteq (\mathsf{Dom} (p_1) \cup \mathsf{Rng} (p_1)) \cap (a_{1},\mathbb{Q}).   
\] 
Furthermore, since $p^i_0$ and $p^i_1$, $0< i < |\bar{a}|$, induce 1 to 1 maps, they agree on $(\mathsf{Dom} (p_0) \cup \mathsf{Rng} (p_0)) \cap (a_{1},\mathbb{Q})$ and in particular, 
\[ 
p^i_1((\mathsf{Dom} (p_0) \cup \mathsf{Rng} (p_0)) \cap (a_{1},\mathbb{Q})) = (\mathsf{Dom} (p_0) \cup \mathsf{Rng} (p_0)) \cap (a_{i},\mathbb{Q}).   
\] 
In order to prove Claim 2 let us consider a triple $p_1 \supseteq p_0 \subseteq p_2$. 
Then for $\bar{a}$ as above we have a triple of $p^{|\bar{a}|}_{i} \in \mathcal{P}^{cap}_{\rho^{|\bar{a}|}}$ considered on $(\mathsf{Dom} (p_i) \cup \mathsf{Rng} (p_i)) \cap (a_{1},\mathbb{Q})$, $i<3$. 
By the amalgamation property of $\mathcal{P}^{cap}_{\rho^{|\bar{a}|}}$ we have some $\mathfrak{p}_{\bar{a}}$
amalgamating these maps. 
Applying the construction of Claim 1 we obtain a partial isomorphism $\hat{p}\in \mathcal{P}^{cap}_{\rho}$ where $\mathfrak{p}_{\bar{a}}$ is the $(a_1, \mathbb{Q})$-part of $\hat{p}^{|\bar{a}|}$. 
\end{proof}

\subsection{The betweennes relation} 
Let $\rho$ be an automorphism of the betweenness relation $(\mathbb{Q}, B)$. 
If we consider $\mathbb{Q}$ together with $\pm \infty$, then there are two cases: $\rho (+\infty) = +\infty$ or 
$\rho (+\infty )  = - \infty$. 
In the first case $\rho$ just belongs to 
$\mathsf{Aut} (\mathbb{Q}, <)$ and 
$\mathcal{C}_{\rho} \subseteq \mathsf{Aut} (\mathbb{Q}, <)$.  
The second case is new. 
In order to show CAP for $\mathcal{P}_{\rho}$ in this case we need some additional analysis. 

From now on each automorphism 
$\gamma \in \mathsf{Aut} (\mathbb{Q}, <)$ is also viewed as an automorphism of the reverse ordering of $\mathbb{Q}$. 
We denote it by $\le^r$.  
In this way each partial orbital of $\gamma$ can be viewed 
as an orbital for $\le^r$. 
It is clear that it has the opposite parity then. 
It is worth noting that if $\rho$ is an automorphism of $(\mathbb{Q}, B)$ then $\rho^2$ belongs to $\mathsf{Aut} (\mathbb{Q}, <)$ and can be viewed also as an automorphism of $\le^r$. 

When $\rho \in \mathsf{Aut} (\mathbb{Q}, B)\setminus \mathsf{Aut} (\mathbb{Q}, <)$ the number of fixed points is bounded by $1$. 
If the fixed point exists we denote it by $q_{\rho}$.  
If it does not exist we define a real number as follows. 
Let 
\[ 
A^r_{\rho} = \{ q\in \mathbb{Q} \, | \, q < \rho (q) \} \,  
\mbox{ and } \, A_{\rho} = \{ q\in \mathbb{Q} \, | \, \rho (q) < q \}. 
\] 
It is easy to verify that $A^r_{\rho} < A_{\rho}$, i.e. we have a Dedekind cut. 
We also consider the pair $(A^r_{\rho},A_{\rho})$ when $q_{\rho}$ exists. 
In this case $\mathbb{Q} = A^r_{\rho} \cup \{ q_{\rho} \} \cup
A_{\rho}$. 

It is worth noting that $\rho (A^r_{\rho}) = A_{\rho}$ and $\rho (A_{\rho})= A^r_{\rho}$. 
In particular $\rho^2$ preserves both $A_{\rho}$ and $A^r_{\rho}$. 
Furthermore, $\rho$ is an isomorphism from $(A_{\rho}, < )$ to $(A^r_{\rho}, <^r )$ which preserves the graph of $\rho^2$. 
Identifying $(A_{\rho}, < )$ with $(1, \mathbb{Q})$ and $(A^r_{\rho}, <^r )$ with $(2, \mathbb{Q})$ we view $\rho$ as an automorphism of the ultrahomogeneous ordering $B_2$ with $\iota_{\rho}=2$.  
By Theorem \ref{B_n} we obtain the following statement. 

\begin{corollary} \label{B_2}
Let 
$\rho \in \mathsf{Aut} (\mathbb{Q}, B)\setminus \mathsf{Aut}(\mathbb{Q},<)$. 
Then the family $\mathcal{P}_{\rho}$ contains a cofinal subfamily with the amalgamation property. 
In particular $\mathcal{C}_{\rho}$ contains a generic automorphism. 
\end{corollary}

\subsection{$B_{\omega}$, $C_n$, $D$ and reducts of $(\mathbb{Q}, C)$}

We state the following theorem without proof 
(the notation is from Section 1.3). 

\begin{theorem} 
Let $M$ be a countable ultrahomogeneous structure which is isomorphic to a partially ordered set $B_{\omega}$ or some $C_n$ with $n\in \omega$. 
Then for any $\rho \in \mathsf{Aut}(M)$ the class $\mathcal{C}_{\rho}$ has a generic element.  
Furthermore, the corresponding $\mathcal{P}_{\rho}$ satisfies CAP. 
\end{theorem}

The proof of this theorem is slightly technical but does not require ideas or methods beyond those ones already used in the paper. 
For example viewing $B_{\omega}$ as $\omega \times \mathbb{Q}$ for any $\rho \in \mathsf{Aut}(B_{\omega})$ one finds a permutation $\sigma_{\rho}$ on $\omega$ exactly as in Section 5.1.  
The further analysis depends on cases of Section 1.2 which correspond to $\sigma_{\rho}$. 
In each one a cofinal CAP subfamily of $\mathcal{P}_{\rho}$ can be found by the method of Theorem \ref{B_n}. 

On the other hand any automorphism $\rho \in \mathsf{Aut}(C_n)$ (with $n\in \omega$) has a natural $\mathbb{Q}$-projection, say $\gamma \in \mathsf{Aut}(\mathbb{Q},<)$.
Note that for each orbital $q\in {\cal O}_{\gamma}$ of parity $0$ the automorphism $\rho$ determines a permutation of $[n]$  on the set of incomparable points of this level.  
The corresponding cycle function, say $f$, can be viewed as a new color $0_f$. 
In order to determine $\mathcal{C}_{\rho}$ we firstly need to know if $\gamma$ is of type $\infty_{+,-,0}$.  
If this is the case it remains to recognize which finite sequences of cycle functions 
$(f_{q_1}, \, \ldots \, f_{q_k})$ can be realized for $q_i \in {\cal O}_{\gamma}$ with $\wp_{\gamma} (q_i ) = 0$.   
If $\gamma$ is not of type $\infty_{+,-,0}$ we have to apply Theorem \ref{decomposition} to 
$({\cal O}_{\gamma}, \prec_{\gamma}, \wp_{\gamma} )$ to determine the corresponding canonical decomposition. 
Applying this theorem to $\infty_0$-parts (possibly several times) we obtain canonical decompositions of them.  
Then the arguments of Sections 3 and 4 can be carried out in a natural way. 
  
The remaining cases need more detailed analysis which requires a separate paper (or papers). 
Indeed, in the case of $C_{\omega}$ we have to consider colored linear orderings with infinitely many colors. Investigation of finite partial isomorphisms of $D$ was initiated by Kuske and Truss in \cite{kuske-truss}. 
The technique used there is much deeper than in the case 
of $(\mathbb{Q},<)$. 
The main result of that paper states that $\mathcal{P}$ satisfies CAP. 
We think that this can be proved for any $\mathcal{P}_{\rho}$ with $\rho \in \mathsf{Aut} (D)$. 
We also think that the same statement holds in the cases of the circular ordering $(\mathbb{Q},C)$ and the corresponding separaion relation.  
These objects look especially interesting.  
It seems likely that the classification of homeomorphisms of the circle (see Chapter 11 in \cite{KH}) is necessary for this result.

%

\end{document}